\documentclass[12pt, a4paper, reqno, oneside]{article}

\usepackage{amssymb,amsmath}
\usepackage{amsmath,accents}
\DeclareUnicodeCharacter{0335}{}
\usepackage{tikz-cd}
\usepackage{authblk} 
\usepackage{physics} 
\usepackage{tikz}
\usepackage{titlesec} 
\usetikzlibrary{babel}
\usepackage{adjustbox}
\usepackage{url}
\usepackage{xcolor}
\usepackage{yhmath}
\usepackage{eso-pic}
\usepackage[T1]{fontenc}
\usepackage[utf8]{inputenc}
\usepackage[spanish,english]{babel}
\usepackage{csquotes}
\usepackage{pst-node}
\usepackage[english]{babel}
\usepackage{latexsym}
\usepackage{euscript}
\usepackage{mathtools}
\usepackage[all]{xy}
\usepackage{mathrsfs}
\usepackage[percent]{overpic}
\usepackage{setspace}
\usepackage{graphicx}
\usepackage{pdfsync}
\usepackage{amsthm}
\usepackage{thmtools}
\usepackage{yfonts}
\usetikzlibrary{calc}
\usetikzlibrary{decorations.pathmorphing}

\tikzset{curve/.style={settings={#1},to path={(\tikztostart)
    .. controls ($(\tikztostart)!\pv{pos}!(\tikztotarget)!\pv{height}!270:(\tikztotarget)$)
    and ($(\tikztostart)!1-\pv{pos}!(\tikztotarget)!\pv{height}!270:(\tikztotarget)$)
    .. (\tikztotarget)\tikztonodes}},
    settings/.code={\tikzset{quiver/.cd,#1}
        \def\pv##1{\pgfkeysvalueof{/tikz/quiver/##1}}},
    quiver/.cd,pos/.initial=0.35,height/.initial=0}

\tikzset{tail reversed/.code={\pgfsetarrowsstart{tikzcd to}}}
\tikzset{2tail/.code={\pgfsetarrowsstart{Implies[reversed]}}}
\tikzset{2tail reversed/.code={\pgfsetarrowsstart{Implies}}}
\tikzset{no body/.style={/tikz/dash pattern=on 0 off 1mm}}

\usepackage[colorlinks=true,linktocpage=true, pagebackref=false, citecolor=black,linkcolor=black]{hyperref}
\normalbaselines
\makeatletter

\def\ps@myfancy{\let\@mkboth\markboth
 \def\@evenhead{\vbox{\hsize\textwidth 
 \hbox to \textwidth{\sf\mdseries\thepage 
 \rule[-.6ex]{0mm}{2mm} \hfill\sf\large\leftmark}
 \vskip 1pt \hrule}}
 \def\@oddhead{\vbox{\hsize\textwidth 
 \hbox to \textwidth{{\sf\large\leftmark}
 \rule[-.6ex]{0mm}{2mm} \hfill\sf\mdseries{\thepage}}
 \vskip 1pt \hrule}}}

\def\ps@myfancyplain{
 \def\@evenhead{\vbox{\hsize\textwidth%
 \rule[-.6ex]{0mm}{2mm} \hfill }
 \vskip 1pt \hrule
 \vskip\headsep
 \vskip\textheight
 \vskip1pc
 \hbox to \textwidth{\sf\mdseries\thepage 
 \rule[-6ex]{0mm}{2mm} \hfill }}
 \def\@oddhead{\vbox{\hsize\textwidth 
 \vskip 1pt\hrule
 \vskip\headsep
 \vskip\textheight
 \vskip2pc
 \hbox to \textwidth{\hfill\rule[.4ex]{1pc}{2.5pt}
 \sf\mdseries\thepage}
}}}

\def\ps@myemptyfun{
 \def\@evenhead{\vbox{\hsize\textwidth
 \rule[-.6ex]{0mm}{2mm} \hfill }
 \vskip 1pt 
 \vskip\headsep
 \vskip\textheight
 \vskip1pc
 \hbox to \textwidth{\sf\mdseries\thepage 
 \rule[-0.6ex]{0mm}{2mm} 
 \hfill }}
 \def\@oddhead{\vbox{\hsize\textwidth 
 \vskip 1pt 
 \vskip\headsep
 \vskip\textheight
 \vskip2pc
}}}

\makeatother
\providecommand{\proofname}{Demostraci\'on.}
 {\par\noindent{\it Demostraci\'on. }\nopagebreak\normalsize}%

 {\par\noindent{\it #1. }\nopagebreak\normalsize}%
 {\hfill\linebreak[2]\hspace*{\fill}$\square$\\[-1pt]}
\makeatother

\def\sqbullet{\raise.2ex\hbox{\vrule width 3.5pt height 3.5pt}}

{}

{\em}

\newcounter{substep}
\def\thesubstep{\arabic{substep}}

{\em}

\newcounter{subsubstep}
\def\thesubsubstep{\arabic{subsubstep}}

{\em}

\numberwithin{figure}{section}

{}

\newtheoremstyle{mystyle}
  {}
  {}
  {\itshape}
  {}
  {\sf \bfseries}
  {}
{ }
  {\thmname{#1}\thmnumber{{\textcolor{blue}{\, \hspace{-1mm}#2.}}}\thmnote{ (#3)}}

\theoremstyle{mystyle}
\definecolor{royalblue(web)}{rgb}{0.25, 0.41, 0.88}
\hypersetup{
    colorlinks=true,
    linkcolor=royalblue(web),
    filecolor=magenta,      
    urlcolor=cyan,
    pdftitle={Overleaf Example},
    pdfpagemode=FullScreen,
    }

\titleformat{\section}
  {\normalfont\LARGE\bfseries}{\thesection.}{1em}{}
\titleformat{\subsection}
  {\normalfont\Large\bfseries}{\thesubsection.}{1em}{}
\titleformat{\subsubsection}
  {\normalfont\normalsize\itshape}{\thesubsubsection.}{1em}{}

\newtheorem{Teor}{Theorem}[section]

\newtheorem{Prop}[Teor]{Proposition}
\newtheorem{Coro}[Teor]{Corollary}

\newtheorem{Defi}[Teor]{Definition}

\newtheorem{Lema}[Teor]{Lemma}

\newtheorem{Obses}[Teor]{Remarks}

\newtheorem{Obse}[Teor]{Remark}

 \newcommand{\N}{{\mathbb N}}
\newcommand{\Z}{{\mathbb Z}} 
\newcommand{\R}{{\mathbb R}}
 \newcommand{\C}{{\mathbb C}}

\newcommand{\Cc}{\mathcal{C}} 

\newcommand{\mail}[1]{\small\href{mailto:#1}{#1}}
\newenvironment{keywords}
{
\begin{center}
\textbf{Keywords}\\
\vspace{0.17cm}
\begin{minipage}{14.5cm}}
{\footnotesize
\end{minipage}
\end{center}}

\newenvironment{Abstract}
{
\begin{center}
\textbf{Abstract}\\
\vspace{0.25cm}
\begin{minipage}{14.5cm}}
{\footnotesize
\end{minipage}
\end{center}}
\definecolor{CorrectionBlue}{RGB}{0,40,210}

\begin{document}


	\begin{center}
		{\huge {\bfseries Generalized Sobolev-Orlicz spaces based on the Riesz fractional gradient as interpolation and potential spaces \par}}
		\vspace{1cm}
		
\begin{tabular}{l@{\hskip 2cm}l} 
	{\Large Pedro Miguel Campos}{\small\textsuperscript{1}} & {\Large Guillermo García-Sáez}{\small\textsuperscript{2}} \\
	\mail{pmcampos@fc.ul.pt} & \mail{guillermo.garciasaez@uclm.es}
\end{tabular}
\vspace{5mm}

\textsc{\textsuperscript{1}CMAFcIO – Departamento de Matem\'atica\\ Faculdade de Ci\'encias, Universidade de Lisboa\\ P-1749-016 Lisboa, Portugal} \\ \vspace{5mm}
\textsc{\textsuperscript{2}ETSII, Departamento de Matem\'aticas\\ Universidad de Castilla-La Mancha} \\
		Campus Universitario s/n, 13071 Ciudad Real, Spain. \\ \vspace{5mm}

Corresponding author: Guillermo Garc\'ia-S\'aez, \mail{guillermo.garciasaez@uclm.es}
\end{center}
\begin{Abstract}
In this work we establish that the recently introduced fractional Sobolev spaces based on the Riesz fractional gradient of Musielak-Orlicz functions by one of the authors, coincide with the space of Bessel potentials of functions on such generalized Orlicz setting. Moreover, we identify them as complex interpolation spaces, and exploiting the well known properties for interpolation of operators we obtain several structural properties for those spaces.
\end{Abstract}
\tableofcontents

\section{Introduction}
In the recent years, nonlocal models have been widely in the study of partial differential equations arising from physical models since they can be used to capture long-range interactions, as well as relax the regularity assumptions over functions. In particular, integro-differential equations driven by the so called \textit{Riesz fractional gradient} $\nabla^s:=\nabla I_{1-s}$, where $I_s$ is the Riesz potential, introduced by Shieh and Spector in \cite{ShiehSpector2015}, have attracted a lot of interest from the commnities of calculus of variations and PDEs. We refer to the introduction of \cite{GarciaSaez2026} for an exhaustive collection of references from the recent years.

The Riesz fractional gradient exhibits many desirable properties from the physical point of view, see \cite{Silhavy2020}, generates the fractional Laplacian, i.e., $-\nabla^s\cdot \nabla^s=(-\Delta)^s$, as well as generalizes the classical gradient $\nabla$ for functions in $W^{1,p}$, i.e., $\nabla^su\to \nabla u$ for every $u\in W^{1,p}$ if $s\to 1^-$ (see \cite{BellidoCuetoMoraCorral2021}). All of these properties suggest that this is the apropiate generalization of the gradient to the fractional/nonlocal setting. Moreover, as it was proved in \cite[Theorem 1.7]{ShiehSpector2015}, the Sobolev-type of space associated to $\nabla^s$ is the well known Bessel potential space $H^{s,p}$, which as it is known, generalizes the classical Sobolev spaces. 

From here, it is natural to study this nonlocal operator on more abstract settings generalizing $L^p$ spaces. This is the idea of the work \cite{campos2024}, in which the first author introduced the fractional Orlicz-Sobolev spaces $H^{s,A}$ in the Musielak, or generalized, Orlicz setting. In the classical case, Musielak-Orlicz-Sobolev spaces $W^{1,A}$ are used to study nonhomogeneous
elastic materials, so it is natural to introduce the nonlocal analogous spaces since it is in solid mechanics where nonlocal gradients have been proven to be more useful (see e.g. \cite{BellidoGarcia2025, Hidde2026})

The purpose of this work is to continue the study of those spaces from the point of view of potential and interpolation theory, which is a natural toolbox if we look at the Bessel potential spaces on the Lebesgue setting, which are at the same time interpolation, nonlocal and potential spaces \cite{BellidoCuetoGarcia2025, BellidoGarcia2025, ShiehSpector2015, ShiehSpector2018}. In particular, we will prove that the spaces $H^{s,A}$ introduced in \cite{campos2024} coincide with the space of Bessel potentials on the Musielak-Orlicz setting, and hence they can be characterized as complex interpolation spaces by means of Calderón-Zygmund theory. Moreover, we will exploit the theory of singular integrals on generalized Orlicz spaces and the well known properties of complex interpolation to obtain several structural result for our spaces. Then main result will be then the identifications 
\[
 X^{s,A}(\mathbb R^n)=H^{s,A}(\mathbb R^n)
 =[L^A(\mathbb R^n),W^{1,A}(\mathbb R^n)]_s,
\]
with equivalent norms. The first identity extends the identification in
\cite[Theorem 1.7]{ShiehSpector2015} to the present setting; the second
identifies its interpolation structure. Finally, as an application, we study an eigenvalue problem on this fractional Orlicz-Sobolev context.
\section{Notation}
We fix $n\in \mathbb{N}$ the dimension of our ambient space $\R^n$ and we will denote by $\Omega\subset \R^n$ and open bounded subset. The notation for Sobolev $W^{1,p}$ and Lebesgue $L^p$ spaces is the standard one, as is that of smooth
functions of compact support $C_c^\infty$. 

\noindent{} We will
indicate the domain of the functions as in $L^p
(\Omega$); the target is indicated only if it is not $\R$.  For two normed spaces $X$, $Y$, we denote the continuous embedding of $X$ into $Y$ as $X\to Y$, i.e., there exists a positive constant $C>0$ such that $\norm{x}_Y\leq C\norm{x}_Y$, for every $x\in X$; The compact embedding between $X$ and $Y$ will be denoted as $X\xhookrightarrow{}\xhookrightarrow{}Y$.

\noindent{}For $\alpha \in\mathbb{N}^n$, we give the standard meaning to the partial derivative $\partial^\alpha$ and the size $|\alpha|$.

\noindent{}Our convention for the Fourier transform of functions $f\in L^1(\R^n)$ is $$\widehat{f}(\xi)=\int_{\R^n}f(x)e^{-2\pi i x\cdot \xi}\,dx,\,\xi\in \R^n.$$ This definition is extended by continuity and duality to other function and distribution spaces as usually in function spaces theory.
The Schwartz space is denoted by $\mathcal{S}$. The variable
in the Fourier space is generically taken to be $\xi$. We will sometimes use the alternative notation $\mathfrak{F}(f)$ for $\widehat{f}$. More details of this operator could be found in the \cite{Grafakos2008}.
\section{Preliminaries}
We collect here several definitions and results that we shall need for our development.
\subsection{Complex interpolation}
In this subsection we recall the definition of the complex method of interpolation and its more important properties.
The complex interpolation method is based on the theory of holomorphic vector valued functions, in particular the ones defined on the strip $$S=\{z\in\C: 0\leq\operatorname{Re}z\leq 1\},$$ taking values on some complex Banach space $E$. Let $(E_0,E_1)$ a couple of Banach spaces such that $E_0,E_1$ are continuously embedded in some Hausdorff topological vector space (compatible couple). We define $\mathfrak{F}\left((E_0,E_1)\right)$ as the space of functions $f:S\to E_0+E_1$ such that $f$ is holomorphic in $\operatorname{int}S$, and continuous and bounded in $S$, and the functions $t\mapsto f(j+it)$ are continuous from $\R\to E_j$, $j=0,1$, and such that $$\norm{f(j+it)}_{E_j}\to 0,\,j=0,1,$$ as $|t|\to \infty$.
Clearly, $\mathfrak{F}(E_0,E_1)$ is a vector space. Moreover, endowed with the norm
$$\norm{f}_{\mathfrak{F}(\overline{E})}:=\operatorname{max}\{\operatorname{sup}_{t\in\R}\norm{f(it)}_{E_0},\operatorname{sup}_{t\in\R}\norm{f(1+it)}_{E_1}\},\,f\in \mathfrak{F}(\overline{E}),$$ it becomes a Banach space. Then, for any $\theta\in [0,1]$, we define $[E_0,E_1]_\theta$ as the space of all $x\in E_0+E_1$ such that there exists $f\in \mathfrak{F}(\overline{E})$ with $f(\theta)=x$, and with the norm 
$$\norm{x}_{\theta}:=\operatorname{inf}\{\norm{f}_{\mathfrak{F}(\overline{E})}: f\in \mathfrak{F}(\overline{E}), f(\theta)=x\}.$$ 
The functor, 
$$\Cc_\theta:(E_0,E_1)\to [E_0,E_1]_\theta,$$ is an exact interpolation functor of exponent $\theta$, i.e., the space $[E_0,E_1]_\theta$ is a Banach space such that $$E_0\cap E_1\xhookrightarrow{}[E_0,E_1]_\theta\xhookrightarrow{}E_0+E_1,$$ and for every operator $T:E_0+E_1\to F_0+F_1$ between two compatible couples, such that 
$$\norm{T}_{E_j\to F_j}=M_j,\,j=0,1,$$ then $T:[E_0,E_1]_{\theta}\to [F_0,F_1]_{\theta}$ with 
$$\norm{T}_{\mathcal{L}\left([E_0,E_1]_{\theta},[F_0,F_1]_{\theta}\right)}\leq M_0^{1-\theta}M_1^\theta.$$ 

The most important properties of complex interpolation spaces are the following:
\begin{Teor}[Properties of Complex interpolation spaces]\label{ComplexProps}
        Let $(E_0,E_1)$ a compatible couple of Banach spaces, and $\theta\in [0,1]$. Then, we have
        \begin{enumerate}
            \item[1.-] $[E_0,E_1]_{\theta}=[E_1,E_0]_{1-\theta}$.
            \item[2.-]If $E_1\xhookrightarrow{}E_0$ and $\theta_0>\theta_1$, $[E_0,E_1]_{\theta_0}\xhookrightarrow{}[E_0,E_1]_{\theta_1}$.
            \item[3.-] If $E_0=E_1$ and $0<\theta<1$, $[E_0,E_1]_\theta=E_0$.
            \item[4.-] $E_0\cap E_1$ is dense in $[E_0,E_1]_\theta$.
            \item[5.-] $[E_0,E_1]_j$ is a closed subspace of $E_j$ with coincidence of the norm in $[E_0,E_1]_j$, $j=0,1$.
            \item[6.-] If $(E_0,E_1)$ is a regular couple and at least one of $E_0$ or $E_1$ is reflexive, then $$\left([E_0,E_1]_\theta\right)^*=[E_0^*,E_1^*]_\theta,\,\theta\in (0,1).$$
            \item[7.-]  If at least one of $E_0$ or $E_1$ is reflexive, then the space $[E_0,E_1]_\theta,\,\theta\in (0,1),$ is reflexive
            \item[8.-] Let $s_0,s_1,\theta\in (0,1)$, and define $s(\theta):=(1-\theta)s_0+\theta s_1$. Then, $$\left[[E_0,E_1]_{s_0},[E_0,E_1]_{s_1}\right]_\theta=[E_0,E_1]_{s(\theta)}.$$
            \item[9.-] Let $T$ a bounded linear operator such that $T:E_0\to F_0$ continuously and $T:E_1\to F_1$ compactly. If $E_0=E_1$ or $F_0=F_1$ or $E_1$ is UMD, then $T:[E_0,E_1]_\theta\to [F_0,F_1]_\theta$ compactly.
         \end{enumerate}
\end{Teor}
Detailed proofs of these facts can be found in \cite[Theorem~4.2.1, Theorem~4.2.2, Theorem~4.5.1]{BerghLofstrom1976} and \cite[Proposition~IV.1.8, Theorem~IV.5.4, Theorem~IV.5.6]{GarciaSaez2024}.

\subsection{Generalized Orlicz spaces}
In this subsection we introduce the main framework for Generalized Orlicz spaces. The definitions are taken from \cite{campos2024}, which in turn are taken from \cite{harjuleto2019, harjuleto2023}.
\begin{Defi}\em
Let $\Omega\subseteq \R^n$ an open set. 
     A function $A:\Omega\times [0,\infty)\to [0,\infty)$ is called a \textit{(generalized) Young function}, and we will denote it as $A\in \Phi(\Omega)$, if:\begin{itemize}
    \item The function $x\mapsto A\left(x,f(|x|)\right)$, is measurable in $\Omega$ for every measurable function $f$ defined on $\Omega$.
    \item $A(x,\cdot)$ is increasing, convex and satisfies $$A(x,0)=0,\,\lim_{t\to 0^+}A(x,t)=0,\,\lim_{t\to \infty}A(x,t)=\infty,$$ a.e. in $\Omega$.
\end{itemize}
\end{Defi}
\begin{Defi}\em
    Let $\Omega\subseteq\R^n$ and open set and $A\in \Phi(\Omega)$. The \textit{left inverse function} of $A$, denoted as $A^{-1}$, is defined as $$A^{-1}(x,t):=\operatorname{inf}\{s\geq 0:A(x,s)\geq t\}.$$
\end{Defi}
We will require some extra hypothesis over the $\Phi$-functions which arise naturally in the Orlicz-spaces theory. We collect them. Let $p>1,q<\infty$. We say that $A$ satisfies:
\begin{itemize}
    \item[$(Inc)_p$] The function $t\mapsto t^{-p}A(x,t)$ is increasing for a.e. $x\in \Omega$;
    \item[$(Dec)_q$] The function $t\mapsto t^{-q}A(x,t)$ is decreasing for a.e. $x\in \Omega$;
    \item[(A0)] There exists a constant $\beta\in (0,1]$ such that $$A(x,\beta)\leq 1\leq A(x,1/\beta),$$ for a.e. $x\in \Omega$;
    \item[(A1)] There exists a constant $\beta\in (0,1]$ such that $$\beta A^{-1}(x,t)\leq A^{-1}(y,t),$$ for every $1\leq t\leq 1/|B|$, a.e. $x,y\in B\cap \Omega$ and every ball $B$ with $|B|\leq 1$;
    \item[(A2)] If for every $s>0$ there exists a constant $\beta\in (0,1]$ and $f\in L^1(\Omega)\cap L^\infty(\Omega)$, with $f\geq 0$, such that for a.e. $x,y\in \Omega$, $$\beta A^{-1}(x,t)\leq A^{-1}\left(y,t+f(x)+f(y)\right),$$ whenever $0\leq t\leq s$.
\end{itemize}
Hypothesis $(Inc)_p$, $(Dec)_q$ are related to the growth of the $\Phi$-functions. Note that the notions of $(Inc)_p$ and $(Dec)_q$ are not invariant under equivalence of $\Phi$-functions, unlike the weaker notions $(aInc)_p$ and $(aDec)_q$, which changes the hypothesis to almost increasing and decreasing, respectively. However, any $(aInc)_p$ or $(aDec)_q$ function could be improved to $(Inc)_p$ or $(Dec)_q$, respectively, by \cite[Lemma 2.2]{harjuleto2016}. If we do not need to specify the values $p$ and $q$, we will say that a function $A\in \Phi(\Omega)$ is $(Inc)$ (respectively $(Dec)$) if it is $(Inc)_r$ (respectively $(Dec)_r)$ for some $\infty>r>1$.
These conditions are also related with the well known $\Delta_2$-condition, also known as \textit{doubling} condition. We say that a $\Phi$-function $A$ is doubling if there exists a positive constant $C$ such that $$A(x,2t)\leq CA(x,t),\,\text{a.e.}\,x\in \Omega,\,t\geq 0.$$ In fact, it is proven in \cite[Lemma 2.2.6(b), Corollary 2.4.11]{harjuleto2019} that $A$ is doubling if and only if is $(Dec)_q$ for some $q<\infty$, and that $A$ satisfies $(Inc)_p$ for some $p>1$ if and only if its conjugate function is doubling. Here, we define the conjugate function of $A$ as $$A^*(x,t):=\sup_{s\geq 0}\{st-A(x,s)\}.$$

Assumptions $(A1)$ and $(A2)$ are technical conditions in order to apply results from harmonic analysis to Orlicz-spaces theory. Conditions $(A0)$ is also relevant as imposes us to work on some sort of unweighted $\Phi$-functions. In fact, it is equivalent to having that $A^{-1}(x,1)\approx 1$. We say that this is related with unweighted functions since if we take for example the function $A(x,t)=w(x)B(t)$, where $B$ is a classical Young function and $w$ is a weight, condition $(A0)$ imposes that $w$ must be comparable with $1$.

Under these hypothesis over the $\Phi$-functions, we introduced the \textit{generalized Orlicz spaces}.
\begin{Defi}
    Let $\Omega\subseteq\R^n$ an open set and $A\in \Phi(\Omega)$. We define tha \textit{modular} $$J_A(f):=\int_\Omega A(x,|f(x)|)\,dx.$$ We define the \textit{generalized Orlicz space} $L^A(\Omega)$, as the space of measurable functions $f:\Omega\to \R$ such that for some $\rho>0$, $$J_A(\rho f)<\infty.$$ The space becomes a normed space endowed with the so called \textit{Luxemburg norm} $$\norm{f}_{L^A(\Omega)}:=\inf \Bigg\{\rho>0:J_A(\rho^{-1}f)\leq 1\Bigg\}.$$
    If the target space is $\R^n$ instead of $\R$, we denote the spaces by $L^A(\Omega;\R^n)$.
\end{Defi}
By \cite[Theorem 2.3.13]{Diening}, generalized Orlicz spaces are Banach spaces. Additionally, the $p$ and $q$ growth conditions gives us a way to compare the classical Lebesgue spaces with the generalized Orlicz spaces. In fact, if $A\in \Phi(\Omega)$ satisfies $(Inc)_p, (Dec)_q$ and $(A0)$, for $1<p<q<\infty$, $$L^p(\R^n)\cap L^q(\R^n)\xhookrightarrow{}L^A(\R^n)\xhookrightarrow{}L^p(\R^n)+L^q(\R^n),$$ i.e., under this assumptions, generalized Orlicz spaces are intermediate spaces with respect to the couple $(L^p,L^q)$. Abstract interpolation theory suggest that there must be an abstract interpolation functor $\mathcal{F}$ such that $\mathcal{F}(L^p,L^q)=L^A$. This is partially true by means of the $\pm$-method as we will see in section 4.

Examples of $\Phi$-functions are the following:
\begin{itemize}
    \item $A(x,t)=\frac{1}{p}t^p$ for $1\leq p<\infty$. The space $L^A(\Omega)$ is just the classical Lebesgue spaces $L^p(\Omega)$.
    \item Let $1<p<\infty$ and $A(x,t)=w(x)t^p$, where $w$ is a weight in the Muckenhoupt class $A_p$ (see \cite{Grafakos2008}). Then, $L^A=L^p_w$, the weighted Lebesgue spaces. Note that since we are dealing with the $(A0)$ condition, those spaces are not under our scope.
    \item Consider a function $p(x):\R^n\to [1,\infty]$ such that $p_-:=\operatorname{ess inf} p(x)>1$ and $p^+:=\operatorname{ess sup}p(x)<\infty$. Then, for $A(x,t)=t^{p(x)}$, the space $L^A$ is the variable exponent Lebesgue space $L^{p(x)}$, which have been widely studied. In this case $A$ satisfies $(aInc)_{p_-}$ and $(aDec)_{p^+}$, and in order to verify $(A1)$ and $(A2)$ we require $p$ to verify that $1/p$ is log-H\"older continuous and satisfies the Nekvinda's condition (see \cite{Uribe2018}).
    \item Let $a(x,t):=pt^{p-2}+qw(x)t^{q-2}$ for some $1<p<q<\infty$, where $w\in L^\infty(\R^n)\cap C^{0,n(q-p)/p}(\R^n)$. Then, $A(x,t):=\int_0^t a(x,s)s\,ds,$ satisfies $(Inc)_p, (Dec)_q, (A0), (A1)$ and $(A2)$ by \cite[Propositions 7.2.1 and 7.2.2]{harjuleto2019}. The space $L^A$ is known as \textit{double phase} Lebesgue space.
\end{itemize}
Concerning the functional analysis of generalized Orlicz spaces, it deeply relies on the hypothesis over the $\Phi$-functions. First of all, concerning the dual of $L^A$, is it not necessarily true that $(L^A(\Omega))^*=L^{A^*}(\Omega)$, which is what we could expect. In \cite[Section 2.7]{Diening} there is a nice discussion the problem of characterization of the dual space of $L^A$. However, under our assumptions, we can ensure that both spaces, $(L^A)^*$ and $L^{A^*}$, coincide. In particular, if $A$ satisfies $(A0)$ and is doubling, we have that by \cite[Theorem 2.7.14]{Diening}, $L^{A^*}(\Omega)=\left(L^A(\Omega)\right)^*.$

Moreover, under the same hypothesis, by \cite[Theorem 4.5]{harjuleto2016}, we have the density of the test functions $C_c^\infty(\Omega)$ on $L^A(\Omega)$; and by \cite[Theorem 4.6]{harjuleto2016}, under the extra hypothesis of $A^*$ being doubling, we have that $L^A(\Omega)$ is reflexive and seperability. Hence, for the reflexivity and separability of $L^A$ is enough to have $A$ satisfying $(A0)$, $(Inc)_p$ and $(Dec)_q$, for some $1<p, q<\infty$.

Analogous to what is done for Lebesgue spaces, we can defined generalized Orlicz-Sobolev spaces. \begin{Defi}
    Let $A\in \Phi(\Omega)$. A function $u\in L^A(\Omega)$ belongs to the \textit{generalized Orlicz-Sobolev space} $W^{1,A}(\Omega)$ if its weak gradient $\nabla u$ belongs to $L^A(\Omega)$. We define a modular $J_{W^{1,A}}$ on $W^{1,A}(\Omega)$ as $$J_{W^{1,A}}(u):=\int_\Omega A(x,|u(x)|)\,dx+\int_\Omega A(x,|\nabla u(x)|)\,dx,$$ which induces the norm $$\norm{u}_{W^{1,A}}:=\inf\Big\{\rho>0: J_{W^{1,A}}(\rho^{-1}u)\leq 1\Big\},$$ which is equivalent to the norm $$\norm{u}_{L^A}+\norm{\nabla u}_{L^A},$$ where $\norm{\nabla u}_{L^A}$ must be understood as $\norm{|\nabla u|}_{L^A}.$
\end{Defi}
As closed subspaces of $L^A$, generalized Orlicz-Sobolev inherit most of the properties from $L^A$ under the same assumptions over the function $A$. Hence, if $A$ is $(A0)$, $(Inc)_p$ for some $p>1$ and $(Dec)_q$ for some $q<\infty$, we have that $W^{1,A}(\Omega)$ is reflexive and separable. 
For the density of test functions, we have to impose the harmonic analysis related assumptions. In particular, if $A$ satisfies $(A0)-(A2)$, and is doubling, $C_c^\infty(\R^n)$ is dense in $W^{1,A}(\R^n)$ \cite[Theorem 6.5]{harjuleto2016}.

Now, if we define $W^{1,A}_0(\Omega)$ as the closure of $C_c^\infty(\Omega)$ for the $W^{1,A}(\Omega)$-norm, we have that by \cite[Theorem 6.8, Lemma 6.9, Theorem 6.10]{harjuleto2016} and \cite[Lemma 6.1.10]{harjuleto2019}:
\begin{itemize}
    \item If $A$ satisfies $(A0)$, $(Inc)_p, (Dec)_q$, then $W^{1,A}_0(\Omega)$ is reflexive and separable.
    \item If $A$ satisfies $(A0)$ and $(Inc)_p$, $W^{1,A}_0(\Omega)\xhookrightarrow{}W^{1,p}_0(\Omega)$.
    \item If $u\in W^{1,A}_0(\Omega)$, the extension by zero of $u$ to the whole $\R^n$ belongs to $W^{1,A}(\R^n)$.
    \item If $u\in W^{1,A}(\Omega)$ with support contained in $\Omega$, $u\in W^{1,A}_0(\Omega)$.
\end{itemize}
Another key property of generalized Orlicz spaces that we will need is the so called \textit{UMD} property. For details on UMD spaces we refer to \cite{Burkholder, Hytonen, Hytonen2, Rubio}. The important detail here is that most of the widely studied spaces on PDEs are UMD and hence it is the easiest way to prove that an operator is compact when acting as an admisible operator with one-sided compactness, as it is stated in Theorem \ref{ComplexProps}(9). It is well known that classical Orlicz spaces are UMD if and only if they are reflexibe. For generalized Orlicz spaces we require some extra assumptions that fall under our scoop. The following result was proven in \cite{Lindemulder}:
\begin{Teor}\label{UMD}
    Let $\Omega\subseteq \R^n$ an open set and $A\in \Phi(\Omega)$ such that both $A$ and its conjugate $A^*$ are doubling. Then, the space $L^A(\Omega)$ has the UMD property. Moreover, $W^{1,A}_0(\Omega)$ has the UMD property.
\end{Teor}
We now recall several results that resembles the classical Sobolev theory in the Orlicz setting, that will be very useful for our purposes.
\begin{Prop}[Poincar\'e inequality]\label{PoincareLocal}
    Let $\Omega$ a bounded domain and $A\in \Phi(\Omega)$ satisfying $(A0)$ and $(A1)$. Then, for every $u\in W^{1,A}_0(\Omega)$, $$\norm{u}_{L^A(\Omega)}\leq \operatorname{diam}(\Omega)\norm{\nabla u}_{L^A(\Omega)}.$$
\end{Prop}
\noindent{}\textbf{Proof:} See \cite[Theorem 6.2.8]{harjuleto2019}. \qed

This proposition yields $\norm{\nabla u}_{L^A(\Omega)}$ is an equivalent norm to $W^{1,A}_0(\Omega)$, and combined with the relationships of generalized Orlicz spaces with Lebesgue spaces, it implies the following:

\begin{Coro}\label{inclusionesLocal}
    Let $A\in \Phi(\Omega)$ satisfying $(A0), (A1), (Inc)_p,$ and $(Dec)_q$ for some $1<p\leq q<\infty$. Then, $$W^{1,q}_0(\Omega)\xhookrightarrow{}W^{1,A}_0(\Omega)\xhookrightarrow{}W^{1,p}_0(\Omega).$$
\end{Coro}

Also, a Sobolev embedding is available:
\begin{Prop}[Sobolev embedding]\label{SobolevEmbedding}
    Let $\Omega$ be a bounded domain with Lipschitz boundary. Assume that $A\in \Phi(\Omega)$ satisfies $(A0)-(A2)$ and $(Dec)_q$ for some $q<n$. Then, if $B\in \Phi(\Omega)$ satisfies that $t^{-1/n}A^{-1}(x,t)\approx B^{-1}$, then $$W^{1,A}_0(\Omega)\xhookrightarrow{}L^B(\Omega).$$
\end{Prop}
\noindent{}\textbf{Proof:} See \cite[Corollary 6.3.4]{harjuleto2019}. \qed

And finally, an Orlicz version of the Rellich-Kondrachov theorem:
\begin{Prop}[Compact embedding]\label{CompactEmbedding}
    Let $A\in \Phi(\R^n)$ satisfying $(A0), (A1)$ and $(Dec)_q$. Then, $$W^{1,A}_0(\Omega)\xhookrightarrow{}\xhookrightarrow{}L^A(\Omega).$$ Moreover, if $q<n$ and $B\in \Phi(\R^n)$ verifies $t^{-\alpha}A^{-1}(x,t)\approx B^{-1}(x,t)$, for some $\alpha\in [0,1/n)$, then $$W^{1,A}_0(\Omega)\xhookrightarrow{}\xhookrightarrow{}L^B(\Omega).$$
\end{Prop}
\noindent{}\textbf{Proof:} See \cite[Theorems 6.3.7, 6.3.8]{harjuleto2019}. \qed

To end this subsection, we recall the complex interpolation for generalized Orlicz spaces \cite[Theorem 5.1]{Uribe2018}, which will be a key result for our purposes:
\begin{Teor}\label{InterpolationOrlicz}
    Let $A_0,A_1\in \Phi(\Omega)$ satisfying $(A0)$. Then, $$[L^{A_0}(\Omega),L^{A_1}(\Omega)]_\theta=L^{A_\theta}(\Omega),$$ where $$A_\theta^{-1}(x,t)=(A_0^{-1}(x,t))^{1-\theta}(A_1^{-1}(x,t))^\theta,$$ for every $0<\theta<1$.
\end{Teor}
As it would be expected, if we take $$A_0(t,x)=t^{p_0},\,A_1(x,t)=t^{p_1},\,1\leq p_0\leq q_0\leq \infty,$$ i.e., $L^{A_j}(\Omega)=L^{p_j}(\Omega)$ and $A_j^{-1}(x,t)=t^{1/{p_j}}, j=0,1,$ we have that $$[L^{p_0}(\Omega),L^{p_1}(\Omega)]_\theta=[L^{A_0}(\Omega),L^{A_1}(\Omega)]_\theta=L^{A_\theta}(\Omega),$$ where \begin{align*}
    A^{-1}_\theta(x,t)=t^{(1-\theta)/p_0+\theta/p_1},
\end{align*} and hence we recover the classical Riesz-Thorin Theorem.

\subsection{Calder\'on-Zymgund operators and multiplier theory}
A key tool on the theory of nonlocal gradients in its interplay with harmonic analysis is the H\"ormander-Mikhlin multiplier theorem \cite[Theorem 4.23]{abels2012}:
\begin{Teor}\label{Milhinnormal}
    Let $m:\R^n\setminus\{0\}\to \C$ a $(n+2)$-times continuously differentiable function such that $$|\partial^\alpha_zm(z)|\lesssim |z|^{-|\alpha|},$$ for all $z\not =0$ and $|\alpha|\leq n+2$. Then, for every $1<p<\infty$, the operator $T$ defined as $Tu=\mathfrak{F}^{-1}(m\widehat{u}),$ for all $u\in \mathcal{S}(\R^n)$, extends to a bounded linear operator $T:L^p(\R^n)\to L^p(\R^n)$.
\end{Teor}
\begin{Obses}
\end{Obses}
\begin{enumerate}
    \item Theorem \ref{Milhinnormal} also holds under the weaker assumption that $|\partial^\alpha_zm(z)|\lesssim |z|^{-|\alpha|}$, for all $|\alpha|\leq \left \lfloor{n/2}\right \rfloor+1$.
    \item Every $(n+2)$-times continuously differentiable function that is homogeneous of degree $0$, i.e., $m(\lambda z)=m(z)$ for all $\lambda>0$, $z\not =0$, satisfies the hypothesis of Theorem \ref{Milhinnormal}.
\end{enumerate}

On the context of Orlicz spaces, it has been established a H\"ormander-Mihlin type theorem in \cite{bonino2024}.In order to prove results in the generalized setting using multiplier theory we will work with Calder\'on-Zymgund operators.

We recall the definition of C-Z operators:
We say that a linear operator $T$ is a \textit{Calder\'on-Zygmund operator} if there exists a kernel $K:\R^n\times \R^n\setminus\{(x,x):x\in \R^n\}\to \R$, such that for all $\varphi\in C_c^\infty(\R^n)$ and $x\not \in \operatorname{supp}f$, $$T\varphi(x)=\int_{\R^n}K(x,y)\varphi(y)\,dy,$$ with the kernel satisfying that $$|K(x,y)|\lesssim|x-y|^{-n},$$ and that for some $\varepsilon>0$, $$|K(x+h,y)-K(x,y)|+|K(x,y)-K(x,y+h)|\lesssim \frac{|h|^\varepsilon}{|x-y|^{n+\varepsilon}},$$ for every $h\in \R^n$ such that $2|h|\leq |x-y|$.

Those operators are connected with multiplier theory on Orlicz spaces by the two following lemmas:
\begin{Lema}\label{lema1}
    Let $m$ in the hypothesis of Theorem \ref{Milhinnormal}. Then, there exists $k\in C^1(\R^n\setminus\{0\})$ such that $$|\partial ^\alpha_zk(z)|\leq C|z|^{-n-|\alpha|},\,z\not =0,\,|\alpha|\leq 1,$$ and the operator $Tf:=\check{m}*f,$ $f\in\mathcal{S}(\R^n)$, satisfies that $Tf=k*f$, for all $x\not\in\operatorname{supp}f$.
\end{Lema}
\begin{Lema}\label{lema2}
    Let $A\in \Phi(\R^n)$ satisfying $(A0), (A1), (A2), (Inc)$ and $(Dec)$, and $T$ a Calder\'on-Zygmund singular integral operator. Then, $T:L^A(\R^n)\to L^A(\R^n)$, continuously.
\end{Lema}
A proof of Lemma \ref{lema1} can be found in \cite[Proposition 4.27]{abels2012}, while the proof for Lemma \ref{lema2} is found in \cite[Corollary 5.4.3]{harjuleto2019}. It is straightoforward to see that the operator defined by convolution with multipliers on the hypothesis of Theorem \ref{Milhinnormal} are Calder\'on-Zygmund singular integral operators, as it is proven in \cite[Lemma 3]{campos2024}. In fact, if $m$ is in the hypothesis of Theorem \ref{Milhinnormal}, by Lemma \ref{lema1}, $$Tf(x):=\check{m}*f(x)=\int_{\R^n}k(x-y)f(y)\,dy,\,x\not\in \operatorname{supp}f,$$ with $$|\partial ^\alpha_zk(z)|\leq C|z|^{-n-|\alpha|},\,z\not=0,\,|\alpha|\leq 1.$$ If we define the function $K(x,y):=k(x-y)$, taking $\alpha=0$, yields $$|K(x,y)|=|k(x-y)|\leq C|x-y|^{-n}.$$ Now, take $|h|\leq \frac{|x-y|}{2}$. Since in that case $|x-y-th|\geq \frac{|x-y|}{2},$ we have that taking $\alpha=1$,
\begin{align*}
    &|k(x-y)-k(x-y-h)|+|k(x-y)-k(x-y+h)|\\
    &\leq \operatorname{sup}_{t\in [0,1]}|Dk(x-y-th)||h|+\operatorname{sup}_{t\in [0,1]}|Dk(x-y+th)||h|\leq C|x-y|^{-1-n}|h|.
\end{align*}
We have established the following sort of H\"ormander-Mihlin multiplier theorem for generalized Orlicz spaces.
\begin{Teor}\label{Orliczmultiplier}
    Let $m\in C^n(\R^n\setminus\{0\})$ such that $$|\partial^\alpha_\xi m(\xi)|\leq C|\xi|^{-|\alpha|},\,\xi\in \R^n,\,|\alpha|=0,1,\ldots,n,$$ for some positive constant $C$. Then, for any $A\in \Phi(\R^n)$ satisfying $(A0), (A1), (A2), (Inc),\\
    (Dec)$, the operator $$T\varphi:=\check{m}*\varphi,\,\varphi\in \mathcal{S}(\R^n),$$ extends uniquely to a bounded linear operator in $L^A(\R^n)$.
\end{Teor}
It is known that UMD property of a Banach space $E$ implies as well multiplier theorems in the sense of \ref{Milhinnormal} for the space $E$, see \cite[Theorem 5.5.10]{Hytonen} and \cite[Theorem 8.3.9]{Hytonen2}.

\subsection{Nonlocal gradients}
In this subsection, we present the main definitions and properties of the fractional operators introduced by Shieh and Spector in \cite{ShiehSpector2015,ShiehSpector2018}, that give rise to our new framework. Let $u\in C_c^\infty(\R^n)$, $v\in C_c^\infty(\R^n;\R^n)$ and $s\in (0,1)$. We define the \textit{Riesz fractional gradient} of $u$ as $\nabla^s$ as $$\nabla^su(x)=c_{n,s}\int_{\R^n}\frac{u(x)-u(y)}{|x-y|^{n+s}}\frac{x-y}{|x-y}\,dy,\,x\in \R^n,$$ and the \textit{fractional divergence} as $$\operatorname{div}^sv(x)=c_{n,s}\int_{\R^n}\frac{v(x)-v(y)}{|x-y|^{n+s}}\cdot\frac{x-y}{|x-y}\,dy,\,x\in \R^n,$$ where $$c_{n,s}=\frac{\Gamma\left(\frac{n+s+1}{2}\right)}{\pi^{n/2}2^{-s}\Gamma\left(\frac{1-s}{2}\right)},$$ is a normalizing constant. One of the main properties of these objects is that they can be seen as the classical gradient $D$ and divergence of the Riesz potential $$I_su:=I_{s}*u(x),$$ where $$I_s(x):=\frac{1}{\gamma_{n,s}}\frac{1}{|x|^{n-s}},\,0<s<n.$$  of the functions $u$ and $v$. Here, the constant $\gamma_{n,s}:=\frac{\pi^{n/2}2^s\Gamma(s/2)}{\Gamma\left(\frac{n-s}{2}\right)}$, satisfies that $$\gamma_{n,1-s}c_{n,s}=(n+s-1).$$ In particular we have the following result: \begin{Prop}Let $s\in (0,1)$, $u\in C_c^\infty(\R^n)$ and $v\in C_c^\infty(\R^n;\R^n).$ Then, \begin{align*}
    \nabla^su&=D(I_{1-s}u)=I_{1-s}(Du),\\
    \operatorname{div}^sv&=\operatorname{div}(I_{1-s}v)=I_{1-s}(\operatorname{div}v).
\end{align*}
\end{Prop}
\noindent{}\textbf{Proof:} See \cite[Proposition 2.7]{GarciaSaez2026}. \qed

The following proposition collect some of the most important facts about those operators.
\begin{Prop}\label{DsProps}
    Let $s\in (0,1)$, $u\in C_c^\infty(\R^n)$ and $v\in C_c^\infty(\R^n;\R^n)$.
    \begin{itemize}
    \item[(1)] The Fourier transform of $\nabla^su$ is $$\widehat{\nabla^su}(\xi)=\frac{2\pi i\xi}{|2\pi\xi|^{1-s}}\widehat{u}(\xi),$$ and the Fourier transform of $\operatorname{div}^su$ is $$\widehat{\operatorname{div}^su}(\xi)=\frac{2\pi i\xi}{|2\pi\xi|^{1-s}}\cdot \widehat{v}(\xi).$$
        \item[(2)] The following fractional integration by parts formula holds: $$\int_{\R^n}\nabla^su(x)\cdot v(x)\,dx=-\int_{\R^n}u(x)\operatorname{div}^sv(x)\,dx.$$
        \item[(3)] We have the following fractional theorem of calculus: $$u(x)=c_{n,-s}\int_{\R^n}\nabla^su(y)\cdot \frac{x-y}{|x-y|^{n-s+1}}\,dy=I_s(\mathcal{R}\cdot \nabla^su),\,x\in\R^n,$$ where $\mathcal{R}$ is the Riesz transform, defined as a Fourier multiplier with symbol $-i\frac{\xi}{|\xi|}$.
        \end{itemize}
\end{Prop}
For detailed proofs of these facts we refer to \cite{ShiehSpector2015}.

In the proof of Theorem \ref{Equivalence}, we will work with an alternative definition for the Riesz fractional gradient, which is in fact, the one that was originally introduced by Shieh and Spector in \cite{ShiehSpector2015}. Let $u\in L^p(\R^n)$ such that $I_{1-s}u$ is well defined. Then, we define the \textit{distributional Riesz fractional gradient} $D^su=(D^su)_j$ as $$(D^su)_j:=\frac{\partial^su}{\partial x_j^s}:=\frac{\partial}{\partial x_j}I_{1-s}u,\,j=1,\ldots,n,$$ in the sense that $$\left\langle \frac{\partial^su}{\partial x_j^s},\varphi\right\rangle:=-\int_{\R^n}I_{1-s}u\frac{\partial \varphi}{\partial x_j}\,dx=-\left\langle I_{1-s}u,\frac{\partial\varphi}{\partial x_j}\right\rangle,$$ for every $\varphi\in C_c^\infty(\R^n)$. Since by \cite[Theorem 1.2]{ShiehSpector2015}, $D^su=I_{1-s}(Du)=D(I_{1-s}u)$, $D^s$ and $\nabla^s$ coincide at least for smooth compactly supported functions.

\section{Generalized Bessel-Orlicz potential spaces}
In this section we introduce the Bessel-Orlicz potential spaces and we prove that those are equivalent to the generalized Sobolev-Orlicz spaces in the whole space $\R^n$ associated to the Riesz fractional gradient introduced in \cite{campos2024}.

For simplicity, we will use the notation $\left\langle\xi\right\rangle:=\left(1+4\pi^2|\xi|^2\right)^{1/2}, \xi\in \R^n$, sometimes called \textit{japanese bracket}.
\begin{Defi}\em
    Let $s\in \C$, $f\in \mathcal{S}'(\R^n)$ (where $ \mathcal{S}'(\R^n)$ is the space of tempered distributions) and $\xi\in \R^n$. We define the \textit{Bessel potential} $\Lambda_s$ of order $s$ of $f$ as $$\Lambda_{s}f=\mathfrak{F}^{-1}\left(\left\langle\xi\right\rangle^{-s}\widehat{f}(\xi)\right).$$
    The name Bessel potential comes from the fact that for every $s>0$, $$\Lambda_sf=G_s*f,\,f\in \mathcal{S}'(\R^n),$$ where $G_s$ is the \textit{Bessel kernel} defined as $$G_s(x):=\frac{1}{(4\pi)^{s/2}\Gamma(s/2)}\int_0^\infty e^{-t/(4\pi)}e^{-|x|^2\pi/t}t^{(s-n)/2}\,\frac{dt}{t},\,x\in \R^n,$$ which satisfies that $\norm{G_s}_1=1$. 
\end{Defi}
\begin{Defi}\em Let $A\in \Phi(\R^n)$ and $s>0$. We define the \textit{generalized Bessel-Orlicz potential space} $H^{s,A}(\R^n)$ as the space 
    $$H^{s,A}(\R^n):=\{u\in \mathcal{S}'(\R^n): \Lambda_{-s}u\in L^A(\R^n)\},$$ endowed with the norm $$\norm{u}_{H^{s,A}(\R^n)}:=\norm{\Lambda_{-s}u}_{L^A(\R^n)}.$$
\end{Defi}
Note that this space is well defined as a subspace of $L^A(\R^n)$.
\begin{Lema}
    Let $A\in \Phi(\R^n)$ satisfie $(A0), (A1)$, $(A2)$, and $s>0$. Then, $H^{s,A}(\R^n)\xhookrightarrow{} L^A(\R^d)$.
\end{Lema}
\noindent{}\textbf{Proof:} Let $f\in H^{s,A}(\R^n)$. We say that a positive function $\sigma\in L^1(\R^n)$ is a \textit{bell-shaped} function if it is radially decreasing and symmetric. 
    Let $G_s$ be the Bessel kernel when $s>0$. In that case $\Lambda_s f=G_s*f$. Then, $\|G_s\|_{1}=1$, and clearly $G_s$ is radially decreasing and radially symmetric. Hence, by \cite[Lemma 4.4.6]{harjuleto2019}, there exists a positive constant $C$ such that $$\norm{G_s*f}_{L^A(\R^n)}\leq C \norm{f}_{L^A(\R^n)},$$ so, 
   $$
        \|f\|_{L^A(\R^n)}= \|\Lambda_{s}\Lambda_{-s} f\|_{L^A(\R^n)}\leq C\|\Lambda_{-s}f\|_{L^A(\R^n)}=C\|f\|_{H^{s,A}(\R^n)}.\qed $$

In \cite{campos2024}, generalized fractional Orlicz-Sobolev spaces were introduced as:
\begin{Defi}
    Let $A\in\Phi(\R^n)$ and $s\in [0,1]$. We define the space $X^{s,A}(\R^n)$ as $$X^{s,A}(\R^n)=\overline{C^\infty_c(\R^n)}^{\|\cdot\|_{s,A}},$$ where $$\norm{u}_{s,A}:=\norm{u}_{L^A(\R^n)}+\norm{\nabla^s u}_{L^A(\R^n;\R^n)}.$$
\end{Defi}
In order to make clear the action of the Riesz fractional gradient on functions $u\in X^{s,A}(\R^n)$, we have to extend the definition of $\nabla^su$ by continuity. Since by definition, for every $u\in X^{s,A}(\R^n)$ there exists a sequence $\{u_m\}\subset C_c^\infty(\R^n)$ such that $u_m\to u$ in $L^A(\R^n)$ and $\{\nabla^su_m\}$ is Cauchy in $L^A(\R^n,\R^n)$, we define $$\nabla^su:=L^A(\R^n;\R^n)-\lim_{m\to \infty} \nabla^su_m.$$ This characterization allows us to directly establish the duality between the operators $\nabla^s$ and $\operatorname{div}^s$, i.e., the integration by parts formula, also holds for functions in $X^{s,A}(\R^n)$. This as well shows that the definition of $\nabla^su$ is independent of the choice of the Cauchy sequence $\{u_n\}$, and hence the space $X^{s,A}(\R^n)$ is well defined.

Note that for the choice $$A_p(x,t):=\frac{t^p}{p},\,1<p<\infty,$$ we have that $H^{s,A_p}(\R^n)=H^{s,p}(\R^n)$, with $H^{s,p}(\R^n)$ being the Bessel potential space (see \cite{BellidoGarcia2025}), and by \cite[Theorem 1.7]{ShiehSpector2015}, $H^{s,A_p}(\R^n)=X^{s,A_p}(\R^n)$. This equivalence also holds for the more abstract setting of generalized Orlicz spaces. The proof is a direct application of the H\"ormanded-Mihlin theorem:
\begin{Teor}\label{Equivalence}
    Let $A\in \Phi(\R^n)$ satisfying $(A0), (A1), (A2), (Inc)_p, (Dec)_q$, and $s\in (0,1)$. Then, $$X^{s,A}(\R^n)=H^{s,A}(\R^n),$$ with equivalence of the norms.
\end{Teor}
\noindent{}\textbf{Proof:} We first prove the inclusion 
    $H^{s,A}(\R^n)\xhookrightarrow{}X^{s,A}(\R^n)$.
    Let $u\in H^{s,A}(\R^n)$. This means that exists a function $v\in L^A(\R^n)$ such that $u=\Lambda_s v$. Since $H^{s,A}(\R^n)\subset L^A(\R^n)$, we only need to prove that $\nabla^s u\in L^A(\R^N;\R^N)$.
    
    Assume that $v\in C_c^\infty(\R^n)$. Since $v\in C_c^\infty(\R^n)$ and $u\in L^A(\R^n)\subset L^p(\R^n)+L^q(\R^n)$, then for all $\varphi\in C_c^\infty(\R^n)$,
    \begin{align*}
        \left\langle\frac{\partial^s u}{\partial x_j^s}, \varphi\right\rangle
        &=-\left\langle I_{1-s}u,\frac{\partial \varphi}{\partial x_j}\right\rangle
        =-\left\langle I_{1-s}\Lambda_s v,\frac{\partial \varphi}{\partial x_j}\right\rangle\\
        &=-\left\langle \Lambda_s(I_{1-s}v),\frac{\partial \varphi}{\partial x_j}\right\rangle
        =\left\langle \Lambda_s\frac{\partial^s v}{\partial x_j^s},\varphi\right\rangle.
    \end{align*}
    Since $v\in C_c^\infty(\R^n)\subset\mathcal{S}(\R^n)$, the Riesz potential $I_{1-s}$ transforms rapidly decreasing functions into tempered distributions, and both the classical gradient $D$ and the Bessel Potential $\Lambda_s$ transform tempered distributions into tempered distributions, this means that $\frac{\partial^s u}{\partial x_j^s}$ is a tempered distribution, and so, we have the following identity for its Fourier transform:
    \begin{align*}
        &\left\langle\mathfrak{F}\frac{\partial^s u}{\partial x_j^s}, \psi\right\rangle=\left\langle \mathfrak{F}\left(\Lambda_s*\frac{\partial^s v}{\partial x_j^s}\right),\psi\right\rangle=\left\langle \Lambda_s\frac{\partial^s v}{\partial x_j^s},\mathfrak{F}\psi\right\rangle=-\left\langle \Lambda_s(I_{1-s}v),\frac{\partial \mathfrak{F} \psi}{\partial x_j}\right\rangle\\
        &=-\left\langle \Lambda_s(I_{1-s}v),\mathfrak{F}(-2\pi i\xi_j \psi)\right\rangle=\left\langle \mathfrak{F}\Lambda_s\mathfrak{F}I_{1-s}\mathfrak{F}v,2\pi i\xi_j \psi\right\rangle\\
        &=\left\langle\langle\xi\rangle^{-s}|2\pi\xi|^{s-1}\widehat{v}, 2\pi i\xi_j\psi\right\rangle=\left\langle i\frac{\xi_j}{|\xi|}\frac{(2\pi|\xi|)^s}{(1+4\pi^2|\xi|^2)^{s/2}}\widehat{g}, \psi\right\rangle.
    \end{align*}
    By Theorem \ref{Orliczmultiplier}, the operator $T$ with symbol $$m_s(\xi):=\frac{(2\pi|\xi|)^s}{\langle\xi\rangle^s},$$ is bounded in $L^A(\R^n)$. Since $$\frac{\partial^s u}{\partial x_j^s}=-\mathcal{R}_j(\check{m_s}*v),$$ and $\mathcal{R}$ is a Calder\'on-Zymgund operator, by Lemma \ref{lema2},
    \begin{equation*}
        \|\nabla^s u\|_{L^A(\R^n,\R^n)}=\|\mathcal{R}\cdot(\check{m_s}*v)\|_{L^A(\R^n,\R^n)}\leq C\|\check{m_s}*v\|_{L^A(\R^n,\R^n)}\leq C'\|v\|_{L^A(\R^n)}.
    \end{equation*}
    Since $C_c^\infty(\R^n)$ is dense in $L^A(\R^n)$, the inequality can be extended for all $v\in L^A(\R^N)$, and so $\nabla^s u\in L^A(\R^n)$ for all $u\in H^{s,A}(\R^n)$.
    
    To establish the reverse inclusion,
    $X^{s,A}(\R^n)\subset H^{s,A}(\R^n)$, we take $f\in C^\infty_c(\R^n)\cap X^{s,A}(\R^n)\subset X^{s,A}(\R^n)$ with the inclusion being dense by definition.
We need to find a function $g\in L^A(\R^n)$ such that $f=\Lambda_s g$. 
    Let us define the operator $K_s$ whose symbol is given by $\frac{\langle \xi\rangle^{s}}{1+(2\pi|\xi|)^s}$. By Theorem \ref{Orliczmultiplier}, $K_s$ is continuous in $L^A(\R^n)$. Define the function $$g:=K_s\left(f+\sum_{j=1}^n{\mathcal{R}_j\left(\frac{\partial^s f}{\partial x_j^s}\right)}\right).$$ Then
    \begin{align*}
        \|g\|_{L^A(\R^n)}&\leq C\left\|f+\sum_{j=1}^N{\mathcal{R}_j\left(\frac{\partial^s f}{\partial x_j^s}\right)}\right\|_{L^A(\R^n)}\leq C \left(\norm{f}_{L^A(\R^n)}+\norm{\mathcal{R}\cdot \nabla^s f}_{L^A(\R^n)}\right)\\
        &\leq C\|f\|_{X^{s,A}(\R^n)},
    \end{align*}
    which means that $g\in L^A(\R^n)$. Moreover, $f=\Lambda_s g$ since
    \begin{align*}
        \mathfrak{F}\{\Lambda_sg\}(\xi)&=(1+4\pi^2|\xi|^2)^{-s/2}\mathfrak{F}\{K_s\}(\xi)\left(\mathfrak{F}\{f\}(\xi)+\sum_{j=1}^n\frac{-i\xi_j}{|\xi|}(2\pi)^si\xi_j|\xi|^{s-1}\mathfrak{F}\{f\}(\xi)\right)\\
     &=(1+4\pi^2|\xi|^2)^{-s/2}\frac{(1+4\pi^2|\xi|^2)^{s/2}}{1+(2\pi|\xi|)^s}\left(1+\frac{(2\pi)^s|\xi|^2}{|\xi|^{2-s}}\right)\mathfrak{F}\{f\}(\xi)=\frac{1+(2\pi|\xi|)^2}{1+(2\pi|\xi|)^s}\mathfrak{F}\{f\}(\xi)\\
     &=\mathfrak{F}\{f\}(\xi),
    \end{align*}
    in the sense of tempered distributions. Then by density, we get the desired result for $f\in X^{s,A}(\R^n)$.\qed

    We can define the generalized fractional Orlicz-Sobolev spaces for bounded domains extended by zero on $\Omega\subset \R^n$ as $$\widetilde{H}_0^{s,A}(\Omega):=\overline{C_c^\infty(\Omega)}^{X^{s,A}(\R^n)}=\overline{C_c^\infty(\Omega)}^{H^{s,A}(\R^n)}.$$ Clearly, if $\Omega=\R^n$ we have that $H^{s,A}(\R^n)=\widetilde{H}_0^{s,A}(\R^n)$. An straightforward adaptation of the classical methods would yield the identity $$\widetilde{H}_0^{s,A}(\Omega)=\{u\in H^{s,A}(\R^n): u|_{\Omega^c}=0\}:=H^{s,A}_0(\Omega),$$ whenever $\Omega$ has Lipschitz boundary. 
    
   We study now the main embeddings and inequalities for the spaces $\widetilde{H}^{s,A}_0(\Omega)$. The first one is the fractional Sobolev embedding on the generalized Orlicz setting.
    \begin{Teor}[Sobolev inequality]
    Let $p,q>1$ and $s\in(0,1)$ such that $\gamma:=\frac{s}{n}=\frac{1}{p}-\frac{1}{q}$ and $r\in(\gamma,\frac{1}{p}]$. Consider two $\Phi$-functions $A,B\in\Phi(\R^n)$ satisfying (A0), (A1) and (A2). Assume also that $A$ is $inc_p$ and $dec_{1/r}$, $B$ is $inc{q}$ and $dec{\frac{1}{r-\gamma}}$, and that $A^{-1}(x,t)\approx t^\gamma B^{-1}(x,t)$.
    Then, $\widetilde{H}^{s,A}_0(\Omega)\subset L^B(\Omega)$ with the inequality
    \begin{equation*}
        \|u\|_{L^B(\Omega)}\leq C\|\nabla^s u\|_{L^A(\R^d;\R^d)}
    \end{equation*}
\end{Teor}
\begin{proof}
    By \cite[Corollary 5.4.5]{harjuleto2019},
    \begin{equation*}
        \|I_s v\|_{L^B(\R^n)}\leq C\|v\|_{L^A(\R^n)},\quad\forall v\in L^A(\R^d).
    \end{equation*}
    Let $u\in C^\infty_c(\Omega)$. If we set $v:=\mathcal{R}\cdot \nabla^s u$, by the fractional theorem of calculus on Proposition \ref{DsProps}(3) and Lemma \ref{lema2}, $$\norm{u}_{L^B(\Omega)}=\norm{I_sv}_{L^B(\R^n)}\leq C\norm{ \mathcal{R}\cdot \nabla^s u}_{L^A(\R^n)}\leq C'\norm{\nabla^s u}_{L^A(\R^n)},$$ for some positive constant $C'$.
\end{proof}

The next result is a Poincar\'e inequality for $\nabla^s$ in the Orlicz-setting.
\begin{Teor}[Poincaré inequality]\label{thm:poincare_inequality}
    Let $\Omega\subset\R^n$ be a bounded open set, $s\in(0,1)$, and $A\in\Phi(\R^n)$ satisfying $(A0), (A1), (A2)$. Then, there exists a positive constant $C>0$, independent of $s$ and $u$, such that
    \begin{equation*}
        \|u\|_{L^A(\Omega)}\leq \frac{C}{1-2^{-s}}\|\nabla^s u\|_{L^A(\Omega_1)},\quad \forall u\in \widetilde{H}^{s,A}_0(\Omega).
    \end{equation*}
\end{Teor}
\begin{proof}
    By density, it is enough to prove the result for $u\in C_c^\infty(\Omega)$. Consider $r>0$ large enough to have $\Omega\subset B_r(0)$. By the fractional fundamental theorem of calculus in Propositon \ref{DsProps}(2), \begin{align*}u(x)&=c_{n,-s}\int_{\R^n}\nabla^su(y)\cdot \frac{x-y}{|x-y|^{n-s+1}}\,dy\\
&=c_{n,-s}\left(\int_{B_{2r}}\nabla^su(y)\cdot \frac{x-y}{|x-y|^{n-s+1}}\,dy+\int_{B_{2r}^c}\nabla^su(y)\cdot \frac{x-y}{|x-y|^{n-s+1}}\,dy\right),\end{align*} and hence $$\norm{u}_{L^A(\Omega)}\leq c_{n,-s}\left(\norm{\int_{B_{2r}}\frac{|\nabla^su(y)|}{|x-y|^{n-s}}\,dy}_{L^A(\Omega)}+\norm{\int_{B_{2r}^c}\frac{|\nabla^su(y)|}{|x-y|^{n-s}}\,dy}_{L^A(\Omega)}\right).$$
We now proceed to estimate each integral separately. First, let $y\in B_{2r}^c$. By definition, $$|\nabla^su(y)|\leq c_{n,s}\int_{\Omega}\frac{|u(z)|}{|y-z|^{n+s}}\,dz.$$ Since $z\in \Omega\subset B_r$, and $y\in B_{2r}^c$, $|y-z|\geq \frac{|y|}{2}$, so
\begin{align*}
    |\nabla^su(y)|\leq \frac{2^{n+s}c_{n,s}}{|y|^{n+s}}\int_\Omega|u(z)|\,dz\leq \frac{2^{n+s}c_{n,s}}{|y|^{n+s}}\norm{\chi(\Omega)}_{L^{A'}(\Omega)}\norm{u}_{L^A(\Omega)}.
\end{align*}
 Hence, \begin{align*}\int_{B_{2r}^c}\frac{|\nabla^su(y)|}{|x-y|^{n-s}}\,dy&\leq 2^{n+s}c_{n,s}\norm{u}_{L^A(\Omega)}\norm{\chi(\Omega)}_{L^{A'}(\Omega)}\int_{B_{2r}^c}\frac{1}{|y|^{n+s}|x-y|^{n-s}}\,dy\\
&\leq \frac{2^{3-n}c_{n,s}\omega_nr^{-n}}{n}\norm{\chi(\Omega)}_{L^{A'}(\Omega)}\norm{u}_{L^A(\Omega)},\end{align*} thus \begin{align*}\norm{\int_{B_{2r}^c}\frac{|\nabla^su(y)|}{|x-y|^{n-s}}\,dy}_{L^A(\Omega)}&\leq \frac{2^{3-n}c_{n,s}\omega_nr^{-n}}{n}\norm{\chi(\Omega)}_{L^{A'}(\Omega)}\norm{u}_{L^A(\Omega)}\norm{\chi(\Omega)}_{L^A(\Omega)}.\end{align*}
For the other integral we use \cite[Lemma 6.1.4]{Diening}, which says that for every $x\in \R^n$, $\delta>0$, $0<s<n$, and $f\in L_{loc}^1(\R^n)$, we have that $$\int_{B(x,\delta)}\frac{|f(y)|}{|x-y|^{n-s}}\,dy\leq C(s)\delta^s\sum_{k=0}^\infty 2^{-sk}T_{k+k_0}f(x),$$ where $T_k$ is defined as the \textit{averaging operator} $$T_kf:=\sum_{Q\,\text{dyadic}\,: \text{diam}(Q)=2^{-k}}\chi(Q)M_{2Q}f,$$ and $k_0$ is an integer chosen such that $2^{-k_0-1}\leq \delta\leq 2^{-k_0}$. Hence, \begin{align*}
     \int_{B_{2r}}\frac{|\nabla^su(y)|}{|x-y|^{n-s}}\,dy\leq \int_{|x-y|\leq 3r}\frac{|\nabla^su(y)|}{|x-y|^{n-s}}\,dy\leq \frac{\pi^{n/2}}{\Gamma(n/2+1)}(3r)^s\sum_{k=0}^\infty 2^{-sk}T_{k+k_0}|\nabla^su(x)|,
 \end{align*} where $2^{-k_0-1}\leq 3r\leq 2^{-k_0}$. Since by \cite[Theorem 4.4.3]{harjuleto2019}, the family $(T_{k+k_0})_{k}$ is uniformly bounded from $L^A(\Omega)\to L^A(B_{3r})$, there exists a positive constant $C$, not depending on $s$ nor $u$, such that
        \begin{align*}
            \left\|\int_{B_{2r}}{\frac{|\nabla^s u(y)|}{|x-y|^{N-s}}}\,dy\right\|_{L^A(\Omega)}&
            \leq 2^n(3r)^s\sum_{k=0}^\infty{2^{-sk}\|T_{k+k_0}|\nabla^s u|\|_{L^A(\Omega)}}\\
            &\leq Cr^s\sum_{k=0}^\infty{2^{-sk}\|\nabla^s u\|_{L^A(B_{3r})}}\\
            &\leq \frac{Cr^s}{1-2^{-s}}\|\nabla^s u\|_{L^A(B_{3r})}\leq \frac{Cr}{1-2^{-s}}\|\nabla^s u\|_{L^A(B_{3r})}.
        \end{align*}
  Combining both estimates and using the fact that $c_{n,s}$ is uniformly bounded on the parameter $s\in [-1,1]$ (see \cite[Lemma 2.4]{BellidoCuetoMoraCorral2021}), we obtain that there exists a positive constant $C'$, not depending on $s$, such that $$\norm{u}_{L^A(\Omega)}\leq C'\left(r^{-n}\norm{u}_{L^A(\Omega)}+\frac{r}{1-2^{-s}}\norm{\nabla^su}_{L^A(\Omega)}\right).$$ The result now follows choosing $r$ large enough to have $C'r^{-n}\leq 1/2$.
\end{proof}
\begin{Obse}
    If we further assume that the function $A$ is $(aInc)$, the proof simplifies since by \cite[Lemma 6.1.4]{Diening},$$\int_{B(x,\delta)}\frac{|f(y)|}{|x-y|^{n-s}}\,dy\leq C\frac{\delta^s}{1-2^{-s}}Mf(x),$$ and by \cite[Theorem 4.3.4]{harjuleto2019}, $M:L^A(\Omega)\to L^A(\Omega)$ continuously. 
\end{Obse}

Finally, we give the proof of the compact embedding of $\widetilde{H}^{s,A}_0(\Omega)$ into $L^A(\Omega)$.
\begin{Teor}[Fractional Rellich Kondrachov]
    Let $\Omega\subset\R^n$ be an open bounded set and $0<s<1$. Let $A\in\Phi(\R^d)$ satisfy $(A0), (A1), (A2), Inc_{p}$ and $Dec_{q}$ for some $1< p<q<\infty$. Then
    \begin{equation*}
\widetilde{H}^{s,A}_0(\Omega)\xhookrightarrow{}\xhookrightarrow{} L^A(\Omega).
    \end{equation*}
\end{Teor}
\noindent{}\textbf{Proof:}
Fix a nonnegative $\eta\in C_c^\infty(B_1)$ with
$\int\eta=1$, and write
$\eta_\varepsilon(x)=\varepsilon^{-n}\eta(x/\varepsilon)$.
For $u\in C_c^\infty(\Omega)$ set
\[
 v=(-\Delta)^{s/2}u
   =\sum_{j=1}^n\mathcal R_j\partial_j^s u.
\]
The Fourier identities and the fractional fundamental theorem of
calculus give $u=I_sv$; moreover
$\|v\|_{L^A}\leq C\|\nabla^s u\|_{L^A}$.
Define the nonnegative kernel
\[
 K_s(z)=\frac1{\gamma_{n,s}}
       \int_{B_1}\eta(y)
       \left||z-y|^{s-n}-|z|^{s-n}\right|\,dy.
\]
Changing variables in the difference of the two Riesz potentials and
using the triangle inequality yields
\[
 |\eta_\varepsilon*u-u|(x)
 \leq\varepsilon^s(K_{s,\varepsilon}*|v|)(x),
 \qquad
 K_{s,\varepsilon}(x)=\varepsilon^{-n}K_s(x/\varepsilon).
\]
If $0<|z|\leq2$, local integrability of $|\cdot|^{s-n}$ and
boundedness of $\eta$ give $K_s(z)\leq C_{n,s,\eta}|z|^{s-n}$.
If $|z|>2$, the mean value theorem and $|z-ty|\geq|z|/2$
for $0\leq t\leq1$, $|y|\leq1$, give
$K_s(z)\leq C_{n,s,\eta}|z|^{s-n-1}$.
Consequently $K_s$ has a radially decreasing integrable majorant
$C_{n,s,\eta}\Sigma_s$, where
\[
 \Sigma_s(z)=
 \begin{cases}
 |z|^{s-n},&0<|z|\leq2,\\
 2|z|^{s-n-1},&|z|>2.
 \end{cases}
\]
The two pieces agree at $|z|=2$. Its integrability follows from
$s>0$ at the origin and $s<1$ at infinity. Now, \cite[Lemma~4.4.6]{harjuleto2019} therefore implies
\[
 \|\eta_\varepsilon*u-u\|_{L^A(\mathbb R^n)}
 \leq C_{n,s,A,\eta}\varepsilon^s\|v\|_{L^A(\mathbb R^n)}
 \leq C'_{n,s,A,\eta}\varepsilon^s
              \|\nabla^s u\|_{L^A(\mathbb R^n)}.
\]
The constant is independent of $\varepsilon$ and $u$; no uniformity as
$s\to0$ or $s\to1$ is asserted here. Since convolution with
$\eta_\varepsilon$ is bounded on $L^A$ by the same lemma, density
extends this estimate to every $u\in \widetilde{H}^{s,A}_0(\Omega)$.
This extension uses the estimate on test functions, and does not
require defining $I_sv$ for an arbitrary $v\in L^A$.

To finish, a bounded sequence in $\widetilde{H}^{s,A}_0(\Omega)$ has a weakly
convergent subsequence, because the space is reflexive; subtract its
weak limit and denote the resulting sequence by $u_m$.
For fixed $\varepsilon>0$, the functions
$\eta_\varepsilon*u_m$ are uniformly bounded and equicontinuous on
$\overline\Omega$: use boundedness of $u_m$ in $L^1(\Omega)$ and
the bounds on $\eta_\varepsilon$ and its first derivatives.
They converge pointwise to zero by weak convergence in $L^A$, since
$\eta_\varepsilon(x-\cdot)\chi_\Omega$ defines a continuous
functional on $L^A$. Arzel\`a--Ascoli then gives uniform convergence
on $\overline\Omega$, and hence convergence to zero in $L^A(\Omega)$.
The preceding approximation estimate now gives
\[
 \limsup_{m\to\infty}\|u_m\|_{L^A(\Omega)}
 \leq C'\varepsilon^s
     \sup_m\|\nabla^s u_m\|_{L^A(\mathbb R^n)}.
\]
Letting $\varepsilon\to0$ proves the compact embedding. \qed

\section{Interpolation of fractional generalized Orlicz-Sobolev spaces}
In this section, we will study the spaces $H^{s,A}(\R^n)$ as complex interpolation spaces, obtaining analogous results as the ones in the previous section by means of elemental properties of the complex method.
In order to establish that the spaces $H^{s,A}(\R^n)$ are of complex interpolation, we need the following technical lemma.
\begin{Lema}\label{PedroLema} Let $A\in \Phi(\R^n)$ satisfying $(A0), (A1), (A2), (Inc)_p, (Dec)_q$, $u\in L^A(\R^n)$ and $t\in \R$. Then, there exists a positive constant $C$ not depending on $t$ and $u$ such that  
    $$\|\Lambda_{it} f\|_{L^A}\leq C\left\langle t\right\rangle^{n}\|f\|_{L^A}.$$
\end{Lema}
\begin{proof}
    This is a straightforward consequence of Theorem \ref{Orliczmultiplier}, so we only have to prove that $m(\xi)=(1+4\pi^2|\xi|^2)^{-it}$ is on the hypothesis of Theorem \ref{Milhinnormal}.\\
    Let us define the functions $f(r)=r^{-it/2}$, $g(r)=1+4\pi^2r^2$ and $h(\xi)=|\xi|$, hence $m=f\circ g \circ h$. To obtain the estimates for the derivatives of $m$ we apply Faà di Bruno's formula twice (see \cite{Hardy2006}). Let us analyze the derivatives of $w=f\circ g$. For each $k\in\N$ define $\Pi_k=\{(\beta_1,...,\beta_n)\in\N^n:\, \sum_{i=1}^n{i\cdot \beta_i}=k\}$.
    \begin{align*}
        w^{(k)}(r)&=\sum_{\beta\in\Pi_k}{\frac{k!}{\beta!}f^{(|\beta|)}(r)\prod_{\ell=1}^k{\left(\frac{g^{(\ell)}(r)}{\ell!}\right)^{\beta_\ell}}}\\
        &=\sum_{\beta\in\Pi_k}{\frac{k!}{\beta!}\left(\prod_{j=0}^{|\beta|-1}\left(\frac{-it}{2}-j\right)\right)(1+r^2)^{-\frac{it}{2}-|\beta|}\prod_{\ell=1}^{k}{\left(\frac{(1+r^2)^{(\ell)}}{\ell!}\right)^{\beta_\ell}}}
    \end{align*}
    Since $g^{(\ell)}(r)=0$ for all $\ell\geq 3$, then we only care for multiindices $\beta\in\Pi'_k\subset \Pi_k$ where $\Pi'_k=\{(\beta_1, \beta_2, 0, ...,0)\in\N^n:\, \beta_1+2\beta_2=k\}$. Hence,
    \begin{align*}
        w^{(k)}(r)
        &=\sum_{\beta\in\Pi'_k}{\frac{k!}{\beta!}\left(\prod_{j=0}^{|\beta|-1}\left(\frac{-it}{2}-j\right)\right)(1+r^2)^{-\frac{it}{2}-|\beta|}\prod_{\ell=1}^{2}{\left(\frac{(1+r^2)^{(\ell)}}{\ell!}\right)^{\beta_\ell}}}\\
        &=\sum_{\beta\in\Pi'_k}{\frac{k!}{\beta!}\left(\prod_{j=0}^{|\beta|-1}\left(\frac{-it}{2}-j\right)\right)(1+r^2)^{-\frac{it}{2}-|\beta|}2r^{\beta_1}}\\
        &=\sum_{\beta\in\Pi'_k}{\frac{k!}{\beta!}\left(\prod_{j=0}^{|\beta|-1}\left(\frac{-it}{2}-j\right)\right)(1+r^2)^{-\frac{it}{2}-|\beta|}2r^{2|\beta|-n}}.
    \end{align*}
    Moreover, since
    \begin{equation*}
        \left|-\frac{it}{2}-j\right|=\left(\frac{j^2+\frac{t^2}{4}}{1+t^2}\right)^{1/2}(1+t^2)^{1/2}\leq C\left\langle t\right\rangle
    \end{equation*}
    for some $C>0$, we conclude that
    \begin{align*}
        |w^{(k)}(r)|\leq C_{k}\sum_{\beta\in\Pi'_k}{\left\langle t\right\rangle^n\left\langle r\right\rangle^{-2|\beta|}\left\langle r\right\rangle^{2|\beta|-n}}\leq C\left\langle t\right\rangle^n\left\langle r\right\rangle^{-n}.
    \end{align*}
    Finally,
    \begin{align*}
        |\partial^\alpha m(\xi)|&=|
        \partial^\alpha(w\circ h)(\xi)|\leq \sum_{\pi\in\Pi_\alpha}{|w^{(|\pi|)}(|\xi|))|\prod_{\beta\in\pi}{|\partial^\beta g(\xi)|}}\\
        &\leq C\sum_{\pi\in\Pi_\alpha}{\left\langle t\right\rangle^{|\alpha|}\left\langle \xi\right\rangle^{-|\pi|}|\xi|^{|\pi|-|\alpha|}}\leq C\left\langle t\right\rangle^{|\alpha|}|\xi|^{-|\alpha|}.
    \end{align*}
\end{proof}

As a direct consequence of Lemma \ref{PedroLema} and Theorem \ref{Orliczmultiplier}, we get the interpolation identity:
\begin{Teor}\label{Interpolation}
    Let $A$ satisfy $(A0)$--$(A2)$, $(Inc)_p$ and $(Dec)_q$, where
$1<p\leq q<\infty$, and let $0<\theta<1$.Then, $$H^{\theta,A}(\R^n)=[L^A(\R^n),W^{1,A}(\R^n)]_\theta.$$
\end{Teor}
\noindent{}\textbf{Proof:} The proof follows directly from Lemma \ref{PedroLema} with an analogous reasoning as the Lebesgue case in \cite[Theorem 3.7]{BellidoGarcia2025}. See also \cite[Theorem 2.2]{Schumacher} and the references therein. \qed 

Since we are assuming that $A$ is a $\Phi$-function satisfying the unweighted condition $(A0)$, the weighted fractional Sobolev spaces $H^{s,p}_w$, which corresponds with  $A(x,t)=w(x)t^{p}$, $w\in A_p$, from \cite{GarciaSaez2026}, are not under our scop. We leave for a subsequent work how to deal with weighted Orlicz spaces in the fractional setting.

In an analogous manner, we can establish that for $\Omega$ with Lipschitz boundary, $$[L^A(\Omega),W_0^{1,A}(\Omega)]_s=H^{s,A}_0(\Omega),$$ where functions in $H^{s,A}_0(\Omega)$ must be understood as functions in $H^{s,A}_0(\R^n)$ such that $u=0$ in $\Omega^c$. Moreover, if we define $$H^{s,A}(\Omega):=[L^A(\Omega),W^{1,A}(\Omega)]_s,$$ we can argue as in \cite[Theorem 2.18]{BellidoCuetoGarcia2025} using the extension operator for generalized Orlicz-Sobolev spaces given in \cite{Juusti}. In this work, it is proven that if $\Omega$ is a Lipchistz domain (indeed, weaker assumptions on the domain are considered) and $A\in \Phi(\Omega)$ satisfies $(A0)-(A2)$ and $(Dec)_q$ for some $q<\infty$, there exists an extension operator $E:W^{1,A}(\Omega)\to W^{1,\tilde{A}}(\R^n)$, where $\tilde{A}\in \Phi(\R^n)$ satisfies $(A0)-(A2)$ and $(Dec)_q$, and $\tilde{A}|_\Omega=A$. Hence, we get that $$H^{s,A}(\Omega)=H^{s,\tilde{A}}(\R^n)/\sim,$$ where $u\sim v$ if and only if $u|_\Omega=v|_\Omega$; and the interpolation norm being equivalent to the norm $$\norm{u}_{H^{s,A}(\Omega)}=\inf\Big\{\norm{v}_{H^{s,\tilde{A}}(\R^n):}v|_\Omega=u\Big\}.$$

Now that we have established the complex interpolation nature of the fractional generalized Orlicz-Sobolev spaces, we can exploit the properties of interpolation in order to easily establish many results that were obtained in \cite{campos2024} by means of the fractional gradient.  In particular, from \ref{ComplexProps}, the following result is straightforward:
\begin{Prop}
    Let $s\in (0,1)$ and $A\in \Phi(\R^n)$ satisfying $(A0), (A1), (A2), (Inc)_p$ and $(Dec)_q$. The following hold:
     \begin{enumerate}
        \item $H^{s,A}(\R^n),H_0^{s,A}(\Omega)$ and $H^{s,A}(\Omega)$ are complete, reflexive and separable spaces.
        \item $H^{s,A}(\R^n)\xhookrightarrow{}H^{t,A}(\R^n)$, $H^{s,A}_0(\Omega)\xhookrightarrow{}H^{t,A}_0(\Omega)$ and $H^{s,A}(\Omega)\xhookrightarrow{}H^{t,A}(\Omega)$, for every $0<t<s$.
        \item $W^{1,A}_0(\Omega)$ is dense in $H^{s,A}_0(\Omega)$. The space $C_c^\infty(\Omega)$ is dense in $H^{s,A}_0(\Omega)$.
    \end{enumerate}
\end{Prop}
From now on, we will assume that $A\in \Phi(\R^n)$ satisfies $(A0)-(A2)$ and $(Inc)_p$ and $(Dec)_q$, for some $1<p,q<\infty$.

Observe that the embedding $$H^{s,A}_0(\Omega)\xhookrightarrow{}H^{t,A}_0(\Omega),$$ for $s>t$ is, indeed, compact, as it happened for Bessel potential spaces \cite[Proposition 3.1]{BellidoCuetoGarcia2025}.
\begin{Prop}
    Let $0<t<s<1$. Then, we have the compact embedding $$H^{s,A}_0(\Omega)\xhookrightarrow{}\xhookrightarrow{}H^{t,A}_0(\Omega).$$
\end{Prop}
\noindent{}\textbf{Proof:} By Proposition \ref{CompactEmbedding}, $W^{1,A}_0(\Omega)\xhookrightarrow{}\xhookrightarrow{}L^A(\Omega)$. Hence, by \cite[Corollary 2.4]{BellidoCuetoGarcia2025}, $$H^{s,A}_0(\Omega)=[L^A(\Omega),W^{1,A}_0(\Omega)]_s\xhookrightarrow{}\xhookrightarrow{}[L^A(\Omega),W^{1,A}_0(\Omega)]_t=H^{t,A}_0(\Omega),$$ as we wanted to prove. \qed

The dual equivalence of Theorem \ref{Interpolation} can be easily obtained as well. First of all, we will denote by $W^{-1,A^*}(\Omega)$ to the dual of $W^{1,A}_0(\Omega)$. This space can be easily characterized by Proposition \ref{PoincareLocal}. In particular, for any $U\in W^{-1,A^*}(\Omega)$, there exists $\textbf{u}\in L^{A^*}(\Omega;\R^n)$ such that $$\langle U,u\rangle=\int_{\Omega}\textbf{u}\cdot \nabla u\,dx,\,\forall u\in W^{1,A}_0(\Omega).$$ The dual of $H_0^{s,A}(\Omega)$ is denoted by $H^{-s,A^*}(\Omega)$, and can be easily characterized as the space $F\in H^{-s,A^*}(\Omega)$ such that there exists $\textbf{f}\in L^{A^*}(\R^n;\R^n)$, with $$\langle F,g\rangle=\int_{\R^n}\textbf{f}\cdot \nabla^s g\,dx,\,\forall g\in H^{s,A}_0(\Omega).$$
\begin{Prop}
    The following equality holds with equivalence of norms:
    $$H^{-s,A^*}(\Omega)=\left(H^{s,A}_0(\Omega)\right)^*=[L^{A^*}(\Omega),W^{-1,A^*}(\Omega)]_s.$$
\end{Prop}
\noindent{}\textbf{Proof:}
By Theorem \ref{Interpolation}, $$H_0^{s,A}(\Omega)=[L^A(\Omega),W^{1,A}_0(\Omega)]_s,$$ and since both $L^A(\Omega),W^{1,A}_0(\Omega)$ are reflexive and form a regular couple, by Theorem \ref{ComplexProps}(6), we have \begin{align*}
    H^{-s,A^*}(\Omega)&=\left(H^{s,A}_0(\Omega)\right)^*=\left([L^A(\Omega),W^{1,A}_0(\Omega)]_s\right)^*\\
    &=\Big[\left(L^A(\Omega)\right)^*,\left(W^{1,A}_0(\Omega)\right)^*\Big]_s=[L^{A^*}(\Omega),W^{-1,A^*}(\Omega)]_s,
\end{align*}
as we wanted to prove. \qed

In the fractional case, we also have that under the $p,q$-growth assumptions, our space is intermediate with respect two Bessel potential spaces.
\begin{Prop}
    Let $1<p\leq q<\infty$, and $\Omega$ a bounded domain with Lipschitz boundary. Then, we have $$H^{s,q}_0(\Omega)\xhookrightarrow{}H^{s,A}_0(\Omega)\xhookrightarrow{}H^{s,p}_0(\Omega).$$
\end{Prop}
\noindent{}\textbf{Proof:} Since we have that \begin{align*}
    L^q(\Omega)\xhookrightarrow{}&L^A(\Omega)\xhookrightarrow{}L^p(\Omega),\\
    W^{1,q}_0(\Omega)\xhookrightarrow{}&W^{1,A}_0(\Omega)\xhookrightarrow{}W^{1,p}_0(\Omega),
\end{align*}
interpolating the inclusion yields $$H^{s,q}_0(\Omega)=[L^q(\Omega),W^{1,q}_0(\Omega)]_s\xhookrightarrow{}[L^A(\Omega),W^{1,A}_0(\Omega)]_s\xhookrightarrow{}[L^p(\Omega),W^{1,p}_0(\Omega)]_s=H^{s,p}_0(\Omega),$$ and hence the results follows from Theorem \ref{Interpolation}. \qed

Now, we give a Gagliardo-Nirenberg inequality on generalized Orlicz spaces obtained for the Riesz fractional gradient. \begin{Teor}
    Let $0\leq r\leq s\leq t\leq 1$ and  $1<p<\infty$. Then, for every $u\in \widetilde{H}_{0}^{s,A}(\Omega)$, $$\norm{\nabla^su}_{H_{0}^{s,A}(\Omega)}\leq C\norm{\nabla^ru}_{H_{0}^{r,A}(\Omega)}^{1-\theta}\norm{\nabla^tu}_{H_{0}^{t,A}(\Omega)}^\theta,\,\theta=\frac{s-r}{t-r},$$ where $C$ does not depend on $r,s,t$.
\end{Teor}
\noindent{}\textbf{Proof:}
Let us define the function $$m_{t,s}(\xi):=\frac{(2\pi|\xi|)^s}{1+(2\pi|\xi|)^t},\xi\in \R^n.$$ By the computations done in \cite[Lemma 2.3.17]{CamposPHD}, is easy to see that $$\sup_{\alpha\in \N_0:|\alpha|\leq \left[\frac{n}{2}\right]+1}\sup_{\xi\in \R^n\setminus\{0\}}\left|\xi^{|\alpha|}\partial^\alpha_\xi m_{t,s}(\xi)\right|<\infty.$$ Hence, Theorem \ref{Orliczmultiplier} implies that the operator $$\widehat{T\varphi}:=m_{s,t}\widehat{\varphi},\,\varphi\in \mathcal{S}(\R^n),$$ extends to a bounded operator from $L^A(\R^n)\to L^A(\R^n)$. Moreover, note that for every $u\in C_c^\infty(\Omega)$, $$\widehat{(-\Delta)^{s/2}u}=(2\pi |\xi|)^s\widehat{u}=\frac{(2\pi |\xi|)^s}{1+(2\pi|\xi|)^t}\left(1+(2\pi|\xi|)^t\right)\widehat{u}=\mathfrak{F}\Bigg\{T_{t,s}\circ \left(id+(-\Delta)^{t/2}\right)u\Bigg\},$$ and hence $(-\Delta)^{s/2}u=T_{t,s}\circ \left(id+(-\Delta)^{t/2}\right)u$. Now, observe that from the fractional fundamental theorem of calculus in Proposition \ref{DsProps}(3), $$(-\Delta)^{\alpha/2}\nabla^su=(-\Delta)^{(\alpha+s)/2}\mathcal{R}u,\,\alpha\in (0,1).$$ Hence, we can estimate \begin{align*}\norm{(-\Delta)^{s/2}u}_{L^A(\R^n)}&=\norm{T_{t,s}\circ \left(id+(-\Delta)^{t/2}\right)u}_{L^A(\R^n)}\leq C\norm{\left(id+(-\Delta)^{t/2}\right)u}_{L^A(\R^n)}\\
&\leq C \norm{u}_{L^A(\R^n)}+\norm{(-\Delta)^{t/2}u}_{L^A(\R^n)},\end{align*} where the constant $C$ does not depend on $t$ and $s$. Hence, $$\norm{\nabla^su}_{L^A(\R^n)}\leq C \norm{\mathcal{R}u}_{L^A(\R^n)}+\norm{\nabla^tu}_{L^A(\R^n)}\leq C \norm{u}_{L^A(\R^n)}+\norm{\nabla^tu}_{L^A(\R^n)},$$and performing a dilation and optimizing the last expression we find that 
$$\norm{\nabla^su}_{L^A(\R^n;\R^n)}\leq C\norm{u}_{L^A(\R^n;\R^n)}^{\frac{t-s}{t}}\norm{\nabla^tu}_{L^A(\R^n;\R^n)}^{\frac{s}{t}},$$ which proves the case $r=0$. Now, since $$(-\Delta)^{(s-r)/2}\nabla^ru=\mathcal{R}(-\Delta)^{s/2}u=\nabla^su,$$ in an analogous way we can prove the case $r>0$, \begin{align*}
    \norm{\nabla^su}_{L^A(\R^n)}&=\norm{(-\Delta)^{(s-r)/2}\nabla^ru}_{L^A(\R^n)}\leq C \norm{\nabla^ru}_{L^A(\R^n)}+\norm{(-\Delta)^{(t-r)/2}\nabla^ru}_{L^A(\R^n)}\\
&\leq C\norm{\nabla^ru}_{L^A(\R^n)}+C\norm{\nabla^tu}_{L^A(\R^n)},\end{align*} and hence 
$$\norm{\nabla^su}_{L^A(\R^n;R^n)}\leq C \norm{\nabla^ru}_{L^A(\R^n)}^{1-\theta}\norm{\nabla^tu}_{L^A(\R^n)}^\theta,$$ for $\theta=\frac{s-r}{t-r}$. Finally, we can extend the result by density for functions $u\in \widetilde{H}^{s,A}_0(\Omega)$. \qed

By means of the interpolation identity on Theorem \ref{Interpolation} and the complex reiteration theorem \ref{ComplexProps}(8), we can easily derive the result:
\begin{Teor}
    Let $0\leq r\leq s\leq t\leq 1$. For every $u\in W^{1,A}_0(\Omega)$, we have that $$\norm{\nabla^s u}_{L^A(\R^n;\R^n)}\leq C\norm{u}_{L^A(\Omega)}^{1-s}\norm{\nabla u}_{L^A(\Omega;\R^n)}^s.$$ For every $u\in H^{t,A}_0(\Omega)$ we have that $$\norm{\nabla^s u}_{L^A(\R^n;\R^n)}\leq C\norm{\nabla^r u}_{L^A(\R^n;\R^n)}^{\frac{t-s}{t-r}}\norm{\nabla^t u}_{L^A(\R^n;\R^n)}^{\frac{s-r}{t-r}}.$$
\end{Teor}
\noindent{}\textbf{Proof:} Since the complex interpolation is an exact interpolation functor of exponent $\theta$, the inclusion $W^{1,A}_0(\Omega)\xhookrightarrow{}[L^A(\Omega),W^{1,A}_0(\Omega)]_s=H^{s,A}_0(\Omega)$ yields that $$\norm{u}_{H^{s,A}_0(\Omega)}\leq \norm{u}_{L^A(\Omega)}^{1-s}\norm{u}_{W^{1,A}_0(\Omega)},$$ and by the Poincar\'es inequality for $\nabla$ and $\nabla^s$, we get that $$\norm{\nabla^s u}_{L^A(\R^n;\R^n)}\leq C\norm{u}_{L^A(\Omega)}^{1-s}\norm{\nabla u}_{L^A(\Omega;\R^n)}^s,$$ for some positive constant $C>0$. By an analogous reasoning, and the fact that for $\theta=\frac{s-r}{t-r}$, Theorem \ref{ComplexProps}(8) yields that $$[H^{r,A}_0(\Omega),H^{t,A}(\Omega)]_\theta=H^{s,A}_0(\Omega),$$ we have that for every $u\in H^{r,A}_0(\Omega)\cap H^{t,A}_0(\Omega)=H^{t,A}_0(\Omega)$, we have that $$\norm{\nabla^s u}_{L^A(\R^n;\R^n)}\leq C\norm{\nabla^r u}_{L^A(\R^n;\R^n)}^{\frac{t-s}{t-r}}\norm{\nabla^t u}_{L^A(\R^n;\R^n)}^{\frac{s-r}{t-r}},$$ for some positive constant $C>0$. \qed 

This result, together with the well known fact that $$H^{s,p_\theta}_0(\Omega)=[H^{s,p_0}_0(\Omega),H^{s,p_1}_0(\Omega)]_\theta,\,\frac{1}{p_\theta
}=\frac{1-\theta}{p_0}+\frac{\theta}{p_1},$$ makes us wonder if we could obtain a similar inequality but varying the $\Phi$-functions. The key here is that on the whole space, $H^{s,A}(\R^n)$ could be seen as a retract of a vector valued generalized-Orlicz space. First of all, given two Banach spaces $E,F$, we say that $F$ is a \textit{retract} of $E$ if there exist two bounded linear mappings $T:F\to E$ and $G:E\to F$ such that $G\circ T$ is the identity for $F$. The notion of retraction is preserved by interpolation methods, i.e., given two compatible couples $(E_0,E_1)$ and $(F_0,F_1)$ such that $F_j$ is a retract of $E_j$, $j=0,1,$, and $\mathcal{F}$ is some interpolation functor, then the interpolation space $\mathcal{F}(F_0,F_1)$ is a retract of $\mathcal{F}(E_0,E_1)$. Hence, a typical method for finding interpolation identities is to search for a suitable retract for whom interpolation is well known (see \cite[Section 2]{BellidoGarcia2025} for more details on retractions and abstract interpolation theory). 

In the Lebesgue setting, it is known that $H^{s,p}$ is a retract of $L^p(l^s_2)$ \cite[Theorem~6.4.3]{BerghLofstrom1976}, i.e., the space of $p$-summable function taking values in the sequence space $l^s_2$, where $l^s_r$, for $s\in \R$ and $0<r\leq \infty$, is the space of $\R$-valued sequences $(x_m)_m$ such that \begin{align*}\norm{(x_m)}_{\ell^{s}_r(E)}&:=\left(\sum_{m=0}^\infty \left(2^{ms}|x_m|\right)^r\right)^{1/r},\quad r<\infty,\\
    \norm{(x_m)}_{\ell^{s}_\infty(E)}&:=\sup_{m\geq 0} 2^{ms}|x_m|,\quad q=\infty.\end{align*} The proof for the Bessel potential spaces is based on Theorem \ref{Milhinnormal}, and can be found in \cite[Theorem 6.4.3]{BerghLofstrom1976}. An analogous proof can be done for our space $H^{s,A}(\R^n)$ just using Theorem \ref{Orliczmultiplier} instead of Theorem \ref{Milhinnormal}, and hence we obtain that is a retract of $L^A(l^s_2)$ (see also \cite[Theorem 5.5]{GarciaFernandez}). We remark that for a Banach space $E$, we define the $E$-valued Musielak-Orlicz space $L^A(\Omega;E)$ as the space $$L^A(\Omega;E):=\Big\{f:\Omega\to E: \int_\Omega A\left(x,\lambda\norm{f(x)}_E\right)\,dx<\infty,\,\text{for some}\,\lambda>0\Big\}.$$
    As in the Lebesgue space, the proof of Theorem \ref{InterpolationOrlicz} from \cite[Theorem 5.1]{Uribe2018} could be straightforwardly adapted to the vector valued case under the assumption of separability of the space $E$. 
    \begin{Teor}
    Let $A_0,A_1\in \Phi(\Omega)$ satisfying $(A0)$ and $E$ a separable Banach space. Then, $$[L^{A_0}(\Omega;E),L^{A_1}(\Omega;E)]_\theta=L^{A_\theta}(\Omega;E),$$ where $$A_\theta^{-1}(x,t)=(A_0^{-1}(x,t))^{1-\theta}(A_1^{-1}(x,t))^\theta,$$ for every $0<\theta<1$.
    \end{Teor}

    From here, the following results follows directly:
    \begin{Teor}
        Let $A_0,A_1\in \Phi(\R^n)$ under the hypothesis of Theorem \ref{Orliczmultiplier} and $A_s\in \Phi(\R^n)$ such that $$A_\theta^{-1}(x,t)=\left(A_0^{-1}(x,t)\right)^{1-\theta}\left(A^{-1}_1(x,t)\right)^\theta.$$. Then, $$\norm{\nabla^s u}_{L^{A_\theta}}\leq C \norm{\nabla^su}_{L^{A_0}}^{1-\theta}\norm{\nabla^su}_{L^{A_1}}^\theta,$$ for every $u\in H^{s,A_0}_0(\Omega)\cap H_0^{s,A_1}(\Omega)$.
    \end{Teor}
    \noindent{}\textbf{Proof:} Since $H^{s,A_j}(\R^n)$ is a retract of $L^{A_j}(l^s_2)$ for $j=0,1$, we have that $$[H^{s,A_0}(\R^n),H^{s,A_1}(\R^n)]_\theta=[L^{A_0}(l^s_2),L^{A_1}(l^s_2)]_\theta=L^{A_\theta}(l^s_2)=H^{s,A_\theta}(\R^n),$$ and by extension by zero to the whole space, it holds that $$[H_0^{s,A_0}(\Omega),H_0^{s,A_1}(\Omega)]_\theta=H^{s,A_\theta}_0(\Omega),$$ whenever $\Omega$ is a Lipschitz domain. Hence, since $\mathcal{C}_\theta$ is an exact functor of exponent $\theta$, and the Poincaré's inequality on fractional generalized Orlicz-Sobolev spaces, it holds that  $$\norm{\nabla^s u}_{L^{A_\theta}}\leq C \norm{\nabla^su}_{L^{A_0}}^{1-\theta}\norm{\nabla^su}_{L^{A_1}}^\theta,$$ for every $u\in H^{s,A_0}_0(\Omega)\cap H_0^{s,A_1}(\Omega)$, as we wanted to prove. \qed
    
    \begin{Obse}
        We could define \textit{Triebel-Lizorkin} spaces $F^{s}_{A,q}$ on the generalized Orlicz setting in the usual sense as a retract of $L^A(l^s_q)$ for some $1<q<\infty$ adapting the usual proof by means of Theorem \ref{Orliczmultiplier}, and hence we would have the well known identity between Bessel potential and Triebel-Lizorkin spaces on this setting, i.e., $H^{s,A}=F^s_{A,2}$.
    \end{Obse}
 
The following result establishes by interpolation, a fractional Orlicz-Sobolev embedding on the spirit of \cite[Theorem 1]{campos2024}.
\begin{Teor}[Fractional Sobolev embedding]
    Let $A\in\Phi(\R^n)$ satisfying $(A0)-(A2)$, $(Inc)$ and $(Dec)_q$ for some $q<n$. Then, $$H^{s,A}_0(\Omega)\xhookrightarrow{}L^{A_s}(\Omega),$$  where $A_s$ is a $\Phi$-function such that $A_s^{-1}(x,t)\approx t^{-s/n}A^{-1}(x,t)$.
\end{Teor}
\noindent{}\textbf{Proof:} By the Proposition \ref{SobolevEmbedding}, we have that $W^{1,A}_0(\Omega)\xhookrightarrow{}L^B(\Omega)$, where $B^{-1}(x,t)\approx t^{-1/n}A^{-1}(x,t)$. Then, by Theorem \ref{InterpolationOrlicz}, \begin{align*}
    H^{s,A}_0(\Omega)=[L^A(\Omega),W^{1,A}_0(\Omega)]_s\xhookrightarrow{}[L^A(\Omega),L^B(\Omega)]_s=L^{A_s}(\Omega),
\end{align*} where $$A_s^{-1}(x,t)=(A^{-1}(x,t))^{1-s}(B^{-1}(x,t))^s.$$ Since $B^{-1}(x,t)\approx t^{-1/n}A^{-1}(x,t)$, we get that $$A_s^{-1}(x,t)\approx(A^{-1}(x,t))^{1-s}\left(t^{-1/n}A^{-1}(x,t)\right)^s=t^{-s/n}A^{-1}(x,t),$$ as we wanted to prove. \qed

By an analogous reasoning, we directly obtain the following intermediate embeddings:
\begin{Coro}
    Let $A\in\Phi(\R^n)$ under our assumptions and such that $q<n$, and $\gamma\in [0,s]$. Then, $$\widetilde{H}^{s,A}_0(\Omega)\xhookrightarrow{}L^{A_\gamma}(\Omega),$$  where $A_{\gamma}$ is a $\Phi$-function such that $A_\gamma^{-1}(x,t)\approx t^{-\gamma/n}A^{-1}(x,t)$.
\end{Coro}
\begin{Obse}
    Note that if we choose $A(t,x)=t^p$, the latest embedding reads as $$H^{s,p}_0(\Omega)\xhookrightarrow{}L^{A_s}(\Omega),$$ where $$A^{-1}_s(x,t)\approx t^{s-/n}t^{1/p}=t^{\frac{-s}{n}+\frac{1}{p}}=t^{\frac{n-sp}{np}},$$ and hence we recover the subcritial fractional Sobolev embedding from \cite[Theorem 3.21(1)]{BellidoGarcia2025} $$H^{s,p}_0(\Omega)\xhookrightarrow{}L^{\frac{np}{n-sp}}(\Omega).$$ Note that here we have assumed that $p<n$, however in the Sobolev case we know that it is enough to have $sp<n$. In order to prove it by interpolation under this weaker assumption, we would require Orlicz-Sobolev embeddings as in Proposition \ref{SobolevEmbedding} in the cases $q\geq n$.
\end{Obse}
To conclude this section we establish by the same means a compactness result for our spaces: 
\begin{Teor}[Fractional Rellich-Kondrachov]
    Let $A\in \Phi(\R^n)$ under our assumptions and such that $q<n$. Then, we have the compact embedding $$H^{s,A}_0(\Omega)\xhookrightarrow{}\xhookrightarrow{}L^{A_\alpha}(\Omega),$$ where $A^{-1}_\alpha(x,t)\approx t^{-\alpha}A^{-1}(x,t),$ for every $\alpha\in [0,s/n)$.
\end{Teor}
\noindent{}\textbf{Proof:} By Proposition \ref{CompactEmbedding}, $$W^{1,A}_0(\Omega)\xhookrightarrow{}\xhookrightarrow{}L^B(\Omega),$$ where $B^{-1}(x,t)\approx t^{-\beta}A^{-1}(x,t)$, $\beta\in [0,1/n).$ Since generalized Orlicz-Spaces are UMD by Proposition \ref{UMD}, the complex interpolation method preserves one-sided compactness in virtue of Theorem \ref{ComplexProps}(9), and hence $$H^{s,A}_0(\Omega)=[L^A(\Omega),W^{1,A}_0(\Omega)]_s\xhookrightarrow{}\xhookrightarrow{}[L^A(\Omega),L^B(\Omega)]_s=L^{A_s}(\Omega),$$ where $A_s^{-1}(x,t)\approx t^{-s\beta}A^{-1}(x,t)$, for every $\beta\in [0,1/n).$ Hence, again by Theorem \ref{ComplexProps}, $$H^{s,A}_0(\Omega)=[H^{s,A}(\Omega),H^{s,A}_0(\Omega)]_{\theta}\xhookrightarrow{}\xhookrightarrow{}[L^A(\Omega),L^{A_s}(\Omega)]_\theta=L^{A_\alpha}(\Omega),$$ as we wanted to prove. \qed

\section{Further methods of interpolation}
In this section, we will study the continuous embeddings for $H^{s,A}_0(\Omega)$
relaxing the hypothesis of the integrability parameters
whenever $A$ is a $\Phi$-function not depending on $x$, i.e., we will restrict ourselves to
classical Orlicz functions under the assumptions $(A0)-(A2)$, and $(Inc)_p$, $(Dec)_q$ for some
$1<p<q<\infty$. We do not consider generalized Orlicz spaces since our proof relies on
interpolation methods that, as far as we know, have not been adapted to the generalized case.
In particular, we will exploit the properties of an interpolation method developed by García and 
Fernández in \cite{GarciaFernandez} as a variation of the $\pm$-method developed by Gustavsson and
Peetre in \cite{GustavssonPeetre}. Both methods arise in the characterization of classical Orlicz spaces
as interpolation spaces between Lebesgue spaces preserving the UMD property. Throughout this section we will restrict ourselves to functions spaces defined on the whole space, and we will omit this from our notation, i.e., $L^p$ and $H^{s,p}$ will denote $L^p(\R^n)$ and $H^{s,p}(\R^n)$, respectively.

We briefly recall the $\pm$-method, see \cite{GustavssonPeetre} for more details.
Let $\rho:(0,\infty)\to (0,\infty)$. We say that $\rho$ is \textit{pseudoconcave} if it is equivalent
to a concave function. This can be characterized as the existence of a positive constant $C$ such 
that $$\rho(\lambda t)\leq C\max\{1,\lambda\}\rho(t),$$ for every $t>0$ and $\lambda>0$. Now, given
$\rho$ pseudoconcave and $(E_0,E_1)$ a compatible couple of Banach spaces, we define the space 
$\langle (E_0,E_1),\rho\rangle$, as the space of elements $u\in E_0+E_1$ such that there exists 
a sequence $(u_m)_{m\in \Z}\subset E_0\cap E_1$ such that $$u=\sum_{m\in \Z}u_m,$$ with respect to
the convergence in $E_0+E_1$, and that for every finite susbet $J\subset \Z$ and every real sequence 
$|\xi_m|_{m\in J}$ with $|\xi_m|\leq 1$; we have that $$\norm{\sum_{J}\xi_m\frac{2^{mj}}{\rho(2^m)}u_m}_{E_j}\leq C,\,j=0,1,$$
for some $C>0$ constant independent of $(\xi_m)_{m\in J}$ and $J$.
The space is endowed with the semi-norm $$\norm{u}_{\langle E_0,E_1,\rho\rangle}:=\inf_{(u_m)_{m\in \Z}} C,$$ where
the infimum is taken over all admissible sequences $(u_m)_{m\in Z}$ with $C$ having the same significance as above.
The relevant class of functions $\rho$ here is the class denoted as $\mathcal{P}^{\pm}$, which consists
on the pseudoconcave functions $\rho$ such that $$\sup_{t>0}\frac{\rho(\lambda t)}{\rho(t)}=o(\max{1,\lambda}).$$

We now present the variant of the Gustvasson-Peetre method developed in \cite{GarciaFernandez}. Let $\rho\in \mathcal{P}^{\pm}$
and consider the sequence $$\overline{r_m}:=\begin{cases} r_{2m},\,& m\geq 0,\\
r_{2|m|-1},\,& m<0,\end{cases}$$ where $(r_m(\cdot))_{m\in \N}$ is the sequence of Rademacher functions defined as $$t\mapsto r_m(t):=\operatorname{sgn}(\sin 2^{m+1}\pi t);\,t\in [0,1],\,m\in\N,$$ where $\operatorname{sgn}(\cdot)$ denotes the usual sign function.
Now, given a compatible couple of Banach spaces $(E_0,E_1)$ and $1<p<\infty$, we define the space $\langle E_0,E_1\rangle_{\rho,p}$
as the linear space of all $u\in E_0+E_1$ such that there exists a sequence $(u_m)_{m\in\Z}\subset E_0\cap E_1$
which satisfies \begin{align*} &u=\sum_{m\in\Z}u_m,\,\text{(convergence in}\, E_0+E_1),\\
&\sup_{J\subset Z}\int_0^1 \norm{\sum_{m\in J}\overline{r}_m(t)\frac{2^{jm}}{\rho(2^m)}u_m}^p_{E_j}\,dt<\infty,\,j=0,1,
\end{align*}
where the supremum is taken over all finite subsets of $\Z$. In other words, we are assuming that for $k=0,1$, the sequence 
$\left(\overline{r}_m(\cdot)2^{jm}u_m/\rho(2^m)\right)_{m\in \Z}$, is $L^p\left([0,1],E_j\right)$-summable.
We endow the space under the norm $$\norm{u}_{\langle E_0,E_1\rangle_{\rho,p}}:=\inf_{(u_m)_{m\in\Z}}\max_{j=0,1}\sup_{J\subset Z}\norm{\sum_{m\in J}\overline{r}_m(t)\frac{2^{jm}}{\rho(2^m)}u_m}_{L^p\left([0,1],E_j\right)},
$$ where the infimum is taken over all admissible sequences $(u_m)_{m\in \Z}$.

We summarize the main properties of those spaces, see \cite{GarciaFernandez} for detailed proofs:
\begin{Teor}\label{rhoProperties}
    Let $(E_0,E_1)$ a compatible couple of Banach spaces, $\rho\in \mathcal{P}^{\pm}$ and $1\leq p<\infty$. Then the following holds:
    \begin{itemize}
        \item[1)] The space $\langle E_0,E_1\rangle_{\rho,p}$ is complete;
        \item[2)] $\langle E,E\rangle_{\rho,p}=E$ with an equivalent norm for every Banach space $E$;
        \item[3)] If $(F_0,F_1)$ is another compatible couple of Banach spaces such that $E_j\xhookrightarrow{}F_j$, then $\langle E_0,E_1\rangle_{\rho,p}\xhookrightarrow{}\langle F_0,F_1\rangle_{\rho,p}$;
        \item[4)] The functor $(E_0,E_1)\to \langle E_0,E_1\rangle_{\rho,p}$ is an exact interpolation functor, i.e., the space $\langle E_0,E_1\rangle_{\rho,p}$ is an interpolation
space and for every admissible operator $T:(E_0,E_1)\to (F_0,F_1)$ we have that $$\norm{T}_{\langle E_0,E_1\rangle_{\rho,p}\to \langle F_0,F_1\rangle_{\rho,p}}\leq \max\{\norm{T}_{E_0\to F_0},\norm{T}_{E_1\to F_1}\}.;$$
\item[5)] For every $1\leq p,q<\infty$, the spaces $\langle E_0,E_1\rangle_{\rho,p}$ and $\langle E_0,E_1\rangle_{\rho,q}$ are equivalent;
\item[6)] The space $E_0\cap E_1$ is dense in $\langle E_0,E_1\rangle_{\rho,p}$;
\item[7)] If $E_0$ and $E_1$ are $q$-concave Banach lattices for some $q<\infty$, we have that $$\langle L^p(E_0),L^p(E_1)\rangle_{\rho,p}=L^p\left(\langle E_0,E_1\rangle_{\rho,p}\right);$$
\item[8)] Let $A_0,A_1$ two Orlicz functions, $E$ a Banach space and $L^{A_j}(E)$ the corresponding $E$-valued orlicz space for $j=0,1$, where $E$ is a $q$-concave Banach lattice for some $q<\infty$. If $A_j$ and $A_j^*$ are doubling for $j=0,1,$ we have that $$\langle L^{A_0}(E),L^{A_1}(E)\rangle_{\rho,p}=L^{A_\rho}(E),$$ where $$A_\rho^{-1}=A_0^{-1}\rho(A_1^{-1}/A_0^{-1}).$$
\end{itemize}
\end{Teor}
For a given Orlicz function $A$ under our hypothesis, we can get the space $L^A$ as an interpolation space between Lebesgue spaces. In particular, in view of the last theorem, choosing $A_j=t^{p_j}$ for $j=0,1,$ where $1<p_1<p_0<\infty$, and the function $$\rho_{p_0,p_1}(t):=t^{\frac{p_1}{p_1-p_0}}A^{-1}\left(t^{\frac{p_0p_1}{p_0-p_1}}\right),$$ yields that $$\langle L^{p_0},L^{p_1}\rangle_{\rho_{p_0,p_1},p}=L^A,$$ see also \cite[Theorem 4.1]{GarciaFernandez}.

Note that the methods $\langle(\cdot,\cdot),\rho\rangle$ and $\langle\cdot,\cdot\rangle_{\rho,p}$ do not coincide in general, but under certain hypothesis over the Banach couples they both coincide, see \cite[Corollary 5.7]{Garcia1993}.

\begin{Teor}\label{EmbeddingBueno}
    Let $1<p<q<\infty$ such that $sp,sq<n$, $A$ an Orlicz function under our asumptions and $\rho\in \mathcal{P}^{\pm}$ such that $$\rho(t)=t^{\frac{p}{p-q}}A^{-1}\left(t^{\frac{pq}{q-p}}\right).$$ Then, $$H^{s,A}(\R^n)\xhookrightarrow{}L^{A_{\rho,\alpha}}(\R^n),$$ where $$A_{\rho,\alpha}^{-1}(t)=A^{-1}\rho(t^{-\alpha}),$$ for every $\alpha\in [0,s/n]$.
\end{Teor}
\noindent{}\textbf{Proof:} Let $A_0=t^{q}$ and $A_1=t^{p}.$ By Theorem \ref{rhoProperties}(8), we have that $$L^A(l^s_2)=\langle L^{A_0}(l^s_2),L^{A_1}(l^s_2)\rangle_{\rho,2}.$$ Since $H^{s,A}$ is a retract of $L^A(l^s_2)$, we obtain that $$H^{s,A}=\langle H^{s,q},H^{s,p}\rangle_{\rho,2}.$$ Now, by the fractional Sobolev embedding \cite[Theorem 3.21(1)]{BellidoGarcia2025} and Theorem \ref{rhoProperties}(3), $$H^{s,A}=\langle H^{s,q},H^{s,p}\rangle_{\rho,2}\xhookrightarrow{}\langle L^{q_s^*},L^{p_s^*}\rangle_{\rho,2}=L^B,$$ where $$\frac{1}{q_s^*}=\frac{1}{q}-\frac{s}{n},\,\frac{1}{p_s^*}=\frac{1}{p}-\frac{s}{n},$$ and $$B^{-1}(t)=t^{1/q_s^*}\rho\left(t^{1/p_s^*-1/q_s^*}\right).$$ Note that $$\frac{1}{p_s^*}-\frac{1}{q_s^*}=\frac{1}{p}-\frac{1}{q},$$ and hence $$B^{-1}(t)=t^{-s/n}t^{1/q}\rho\left(t^{1/p-1/q}\right)=t^{-s/n}A^{-1}(t).$$ By complex interpolation, we get that $$H^{s,A}\xhookrightarrow{}[L^A,L^B]_\theta=L^{A_\theta},$$ where $$A_\theta^{-1}=(A^{-1})^{1-\theta}(B^{-1})^{\theta}=t^{-\theta s/n}A^{-1}.$$ For every $\alpha\in[0,s/n]$, we can find $\theta\in [0,1]$ such that $\alpha=\theta s/n$. Hence, we denote $A_{\alpha}$ by the associated $A_\theta$. Now, for every $\alpha\in [0,s/n]$, by Theorem \ref{rhoProperties}(2), $$H^{s,A}\xhookrightarrow{}\langle L^A,L^{A_\alpha}\rangle_{\rho,2}=L^{A_{\rho,\alpha}},$$ where $$A_{\rho,\alpha}^{-1}=A^{-1}\rho(B^{-1}/A^{-1})=A^{-1}\rho(t^{-\alpha}),$$ as we wanted to prove. \qed

\section{An eigenvalue problem with  non-standard growth}
As an application of the theory developed we briefly study the following class of fractional elliptic problems
involving non-standard growth conditions. We consider the problem \begin{equation}\label{eqn:Orlicz}
    \begin{cases}
        -\operatorname{div}^s\left(a(x,|\nabla^su|\right)\nabla^s u)&=\lambda|u|^{r-2}u,\,\Omega,\\
        \hfill u&=0,\,\Omega^c
    \end{cases},
\end{equation}
where $\lambda,r>0$ are parameters and $a$ is a measurable and bounded function in $x$ for all $t>0$ such that $a(x,\cdot)$ is Lipschitz continuous a.e. $x\in \R^n$ and that there exists $a_-,a_+>0$ satisfying $$a_-\leq 1+\frac{t\partial_t a(x,t)}{a(x,t)}\leq a_+.$$ We relate this problem with the fractional Orlicz-Sobolev spaces by considering functions $$A(x,t)=\int_0^t a(x,l)l\,dl.$$ Under those hypothesis for $a$, the function $A$ is a Musielak-Orlicz function that is $(Inc)_{a_-+1}$ and $(Dec)_{a_++1}$. Moreover, integrating into the condition on $a$ we obtain that $$(a_-+1)A(x,t)\leq a(x,t)t^2\leq (a_++1)A(x,t),$$ which in particular implies that $a(x,t)>0$ for every $t>0$.

We can prove the following existence result for our problem \ref{eqn:Orlicz}:
\begin{Teor}
    Let $\Omega$ an open bounded set, $0<s<1,$ and $a$ a function in the previous hypothesis. Assume that the previous given $A$ is a function in $\Phi(\R^n)$ satisfying $(A0), (A1), (A2)$, $\lambda>0$ and $1<r<1+a_-$. Then, there exists a non zero solution $u\in \widetilde{H}^{s,A}_0(\Omega),$ to \ref{eqn:Orlicz} in the weak sense.
\end{Teor}
\noindent{}\textbf{Proof:} 
The existence of the solution is a simple consequence of the direct method of
Calculus of Variations. 
Let $$J_{A,\lambda}(u):=J_A(|\nabla^su|)-\frac{\lambda}{r}\int_\Omega|u|^r\,dx,$$ and we fix $p=a_-+1$, $q=a_++1$. An important remark is the fact that for $A\in \Phi(\R^n)$ satisfying $(Inc)_p, (Dec)_q$ for $1<p\leq q<\infty$, $$\frac{1}{2}\min\{\norm{f}_A^p,\norm{f}_A^q\}\leq J_A(f)\leq 2\max\{\norm{f}_A^p,\norm{f}_A^q\}.$$ Then, by Poincar\'e's inequality and the embedding $L^A(\Omega)\xhookrightarrow{}L^r(\Omega),$ we have that $$J_{A,\lambda}(u)\geq \frac{1}{2}\min\{\norm{\nabla^s u}_A^p,\norm{\nabla^s u}_A^q\}-\frac{\lambda C}{r}\norm{\nabla^s u}_A^r,$$ and for $\norm{\nabla^s u}_A$ large enough, since $p\leq q$, $$J_{A,\lambda}(u)\geq \frac{1}{2}\norm{\nabla^s u}_A^p-\frac{\lambda C}{r}\norm{\nabla^s u}_A^r,$$ which proves the coercivity of $J_{A,\lambda}$ in $\widetilde{H}^{s,A}_0(\Omega)$. Finally, we observe that the functional $u\mapsto J_{A}(|\nabla^s u|)$ is convex since $\xi\mapsto A(x,|\xi|)$ is convex a.e. $x\in \R^n$, this along with the compact embedding $\widetilde{H}^{s,A}_0(\Omega)\xhookrightarrow{}\xhookrightarrow{}L^r(\Omega),$ yields the weakly lower semi-continuity of $J_{A,\lambda}$, and hence the existence of a minimizer $u\in \widetilde{H}^{s,A}_0(\Omega)$ for $J_{A,\lambda}$ follows.

We now verify that the minimizer $u$ is a weak solution, i.e., $$\int_{\R^n}a(x,|\nabla^s u|)\nabla^s u\cdot \nabla^s v\,dx=\lambda\int_\Omega|u|^{r-2}uv\,dx,\,\forall v\in \widetilde{H}^{s,A}_0(\Omega).$$
Fix $v\in \widetilde{H}_0^{s,A}(\Omega)$. By convexity,
$(\mathrm{Dec})_q$, and the inequality
$a(x,\tau)\tau^2\le qA(x,\tau)$, for $|t|\le1$ we have
\[
\begin{aligned}
&\left|a(x,|\nabla^s(u+tv)|)
\nabla^s(u+tv)\cdot\nabla^sv\right|\\
&\qquad\le q\,2^{q-1}
\left[A(x,|\nabla^su|)+A(x,|\nabla^sv|)\right].
\end{aligned}
\]
Here $a(x,|\xi|)\xi$ is extended continuously by zero at the
origin, which is possible since $p>1$. The right-hand side
is integrable on $\mathbb{R}^n$. Moreover, since
$u,v\in L^r(\Omega)$,
\[
|u+tv|^{r-1}|v|
\le (|u|+|v|)^{r-1}|v|\in L^1(\Omega).
\]
These bounds justify differentiation under the integral
sign by dominated convergence. Since $t=0$ minimizes
$t\mapsto J_{A,\lambda}(u+tv)$, we obtain
\[
\begin{aligned}
0
&=\left.\frac{d}{dt}J_{A,\lambda}(u+tv)\right|_{t=0}\\
&=\int_{\mathbb{R}^n}
a(x,|\nabla^su|)\nabla^su\cdot\nabla^sv\,dx
-\lambda\int_\Omega |u|^{r-2}uv\,dx.
\end{aligned}
\]
As $v$ was arbitrary, this proves the desired weak
formulation.

It remains to prove that $u\neq0$. Fix
$\varphi\in C_c^\infty(\Omega)\setminus\{0\}$.
By $(\mathrm{Inc})_p$, we have
$A(x,t\tau)\le t^p A(x,\tau)$ for $0<t<1$ and $\tau\ge0$.
Consequently,
\[
J_{A,\lambda}(t\varphi)
\le t^r\left(
t^{p-r}J_A(|\nabla^s\varphi|)
-\frac{\lambda}{r}\int_\Omega|\varphi|^r\,dx
\right).
\]
Since $p>r$, $\lambda>0$ and $\varphi\neq0$, the expression
in parentheses is strictly negative for sufficiently small
$t>0$. By the minimality of $u$, we obtain
\[
J_{A,\lambda}(u)
\le J_{A,\lambda}(t\varphi)
<0=J_{A,\lambda}(0).
\]
Thus $u\neq0$.

\qed

\begin{Obse}
    Let $u\not =0$ any solution to the problem for $\lambda\in\R$. Taking $v=u$ in the weak formulation yields that $$\lambda\int_\Omega|u|^r\,dx=\int_{\R^n} a(x,|\nabla^su|)|\nabla^s u|^2\,dx\geq p\int_{\R^n}A(x,|\nabla^s u|)\,dx=pJ_{A}(|\nabla^s u|).$$ Since $J_{A}(|\nabla^s u|)\geq 0$, and $J_{A}(|\nabla^s u|)=0$ if and only if $\nabla^s u=0$ a.e. in $\R^n$; so if $\lambda\leq 0$, the only solution is $u=0$, which yields a contradiction. 
\end{Obse}

In \cite{CamposPHD}, further properties of the fractional Orlicz Laplacian $$-\Delta_s^A u:=-\operatorname{div}^s\left(a(x,|\nabla^s u|\right)\nabla^s u),$$ were studied. In particular, existence and uniqueness of the Dirichlet problem $-\Delta_A^s u=f$ was studied for functions $f\in H^{-s,A^*}(\Omega)$, as well as the continuous dependence of the solutions on the parameter $s$.

\section*{Acknowledgements}
The authors would like to thank J.C. Bellido for the useful suggestions
on a preliminary version of the manuscript and for hosting the first author during his visit to Universidad de Castilla-La Mancha in February of 2025. The authors would also like to thank O. Domínguez for several conversations on the subject of this paper.
\section*{Author Contributions}
We declare that the contribution of all co-authors was approximately equivalent
\section*{Funding}
This work was supported by {\it Agencia Estatal de Investigación} (Spain) through grant PID2023-151823NB-I00 and {\it Junta de Comunidades de Castilla-La Mancha} (Spain) through grant SBPLY/23/180225/000023. P.M.C. is supported by PhD FCTgrant UI/BD/152276/2021 and G.G.-S. is supported by a Doctoral Fellowship by \textit{Universidad de Castilla-La Mancha} \text{2024-UNIVERS-12844-404}.
\section*{Data Availability}
No datasets were generated or analysed during the current study
\section*{Statements and declarations}
\textbf{Ethical Approval} This article does not contain any studies with human participants or animals performed by
the authors.

\noindent{}\textbf{Competing interests} The authors declare no competing interests.

\appendix


\addcontentsline{toc}{section}{References}
\bibliographystyle{plain} 

\end{document}